\documentclass[11pt, DIV10,a4paper]{article}
\usepackage{natbib}
\newcommand {\ctn}{\citet} 
\usepackage{graphicx,subfigure,amsmath,latexsym,amssymb}
\usepackage{float,epsfig,multirow,rotating,times}
\usepackage{upgreek,wrapfig}
\usepackage{comment}
\usepackage{xr-hyper}

\newcommand{\bzero}{\boldsymbol{0}}

\newtheorem{theorem}{Theorem}

\newtheorem{corollary}[theorem]{Corollary}

\newtheorem{lemma}[theorem]{Lemma}

\newtheorem{remark}[theorem]{Remark}

\newenvironment{proof}[1][Proof]{\textbf{#1.} }{\ \rule{0.5em}{0.5em}}

\numberwithin{equation}{section}
\numberwithin{algo}{section}
\numberwithin{table}{section}
\numberwithin{figure}{section}

\usepackage{setspace}
\usepackage[mathscr]{euscript}
\usepackage[margin=1in]{geometry}
\begin{document}
\normalsize

\title{\vspace{-0.8in}
On Classical and Bayesian Asymptotics in State Space Stochastic Differential Equations}
\author{Trisha Maitra and Sourabh Bhattacharya\thanks{
Trisha Maitra is a postdoctoral fellow and Sourabh Bhattacharya 
is an Associate Professor in
Interdisciplinary Statistical Research Unit, Indian Statistical
Institute, 203, B. T. Road, Kolkata 700108.
Corresponding e-mail: sourabh@isical.ac.in.}}
\date{\vspace{-0.5in}}
\maketitle%

\begin{abstract}
In this article we investigate consistency and asymptotic normality of the maximum likelihood and
the posterior distribution of the parameters in the context of state space stochastic differential
equations ($SDE$s). We then extend our asymptotic theory to random effects models based on systems of
state space $SDE$s, covering both independent and identical and independent but non-identical 
collections of state space $SDE$s.
We also address asymptotic inference in the case of multidimensional linear random effects, and in
situations where the data are available in discretized forms. It is important to note that asymptotic
inference, either in the classical or in the Bayesian paradigm, has not been hitherto investigated
in state space $SDE$s.
\\[2mm]
 
\end{abstract}

\section{Introduction}
\label{sec:intro}

State-space models are well-known for their versatility in modeling complex dynamic systems
in the context of discrete time, and have important applications in various disciplines
like engineering, medicine, finance and statistics. As is also well-known, most time series models
of interest can be expressed in the form of state space models; see, for example, \ctn{Durbin01}
and \ctn{Shumway11}. 
Discrete time state space models are characterized by a latent, unobserved stochastic process, 
$X=\left\{X(t);t=0,1,2,\ldots\right\}$ and another stochastic process $Y=\left\{Y(t);t=0,1,2,\ldots\right\}$, 
the distribution
of which depends upon $X$. The observed time series data are modeled by the conditional distribution
of $Y$ given $X$, where $X$ is assumed to have some specified distribution. An important special case
of such discrete state space models is the hidden Markov model. Here $X$ is assumed to be a Markov chain,
the distribution of $Y(t)$ depends upon $X(t)$, and conditionally on $X(t)$'s, $Y(t)$'s are independent. 
Such models have important applications in engineering, finance, biology, statistics; see, for example,
\ctn{Elliott95} and \ctn{Cappe05}.


However, when the time is continuous, research on state space or hidden Markov models seem to be
much scarce. Ideally, one should consider a pair of stochastic differential equations ($SDE$s)
whose solutions would be the continuous time processes $Y=\left\{Y(t):t\in[0,T]\right\}$ and 
$X=\left\{X(t):t\in[0,T]\right\}$.
In fact, the $SDE$ with solution $Y$ should depend upon $X$. Since the solutions of $SDE$s under general
regularity conditions are Markov processes (see, for example, \ctn{Mao11}), $X$ would turn out to be
a Markov process, and conditionally on $X$, $Y$ would also be a Markov process. Thus, such an approach could be
interpreted as continuous time versions of the traditional discrete time hidden Markov model based approach.
Continuous time models closely resembling the above-mentioned type exists in the literature, but rather than
estimating relevant parameters, filtering theory has been considered. For instance,
\ctn{Stratonovich68}, \ctn{Jazwinski70}, \ctn{Maybeck79}, \ctn{Maybeck82},
\ctn{Sarkka06}, \ctn{Crisan11} consider the filtering problem in state space $SDE$s of the following type:
\begin{align}
d Y(t)&=b_Y(X(t),t)dt+dW_Y(t);
\label{eq:data_sde0}\\
d X(t)&=b_X(X(t),t)dt+\sigma_X(X(t),t)dW_X(t),
\label{eq:state_space_sde0}
\end{align}
where $W_Y$ and $W_X$ are independent standard Wiener processes, $b_Y$, $b_X$ are real-valued
drift functions, and $\sigma_X$ is the real-valued diffusion coefficient.
The $SDE$s are assumed to satisfy the usual regularity conditions that guarantee existence
of strong solutions; see, for example, \ctn{Arnold74}, \ctn{Oksendal03}, \ctn{Mao11}.
The purpose of filtering theory is to compute the posterior distribution of the latent process
conditional on the observed process. 
This can be obtained from the the continuous-time optimal filtering equation, which 
is, in fact, the Kushner-Stratonovich equation (\ctn{Kushner64}, \ctn{Bucy65}). Note that
(see \ctn{Sarkka12b}, for example) it is possible to obtain the latter as continuous-time limits of the Bayesian 
filtering equations. The so-called Zakai equation (\ctn{Zakai69}) provides a simplified form 
by removing the non-linearity in the Kushner-Stratonovich equation. In the special case of (\ref{eq:data_sde0})
and (\ref{eq:state_space_sde0}), with $b_Y(X(t),t)=L(t)X(t)$, $b_X(X(t),t)=H(t)X(t)$ and $\sigma_X(X(t),t)=\sigma_X(t)$,
exact solution of the filtering problem, known as the Kalman-Bucy filter (\ctn{Kalman61}), can be obtained.
In the non-linear cases various approximations are employed; see \ctn{Crisan11}, \ctn{Sarkka07}, \ctn{Sarkka13},
among others.

In pharmocokinetic/pharmacodynamic contexts, the following type of model is regarded
as the state space model, assuming $\left\{Y_1,\ldots,Y_n\right\}$ are observed at discrete 
times $\left\{t_1,\ldots,t_n\right\}$:
\begin{align}
&Y_j = b_Y(X_{t_j},\theta)+\sigma_Y(X_{t_j},\theta)\epsilon_j;\quad \epsilon_j\stackrel{iid}{\sim}N\left(0,1\right);
\label{eq:data_sde_pkpd}\\
&d X(t)=b_X(X(t),t,\theta)dt+\sigma_X(X(t),\theta)dW_X(t),
\label{eq:state_space_sde_pkpd}
\end{align}
where $b_Y$ and $\sigma_Y$ are appropriate real-valued functions, and $\theta$ denotes the set
of relevant parameters. The standard choices of $\sigma_Y$ 
are $\sigma_Y(x,\theta)=\sigma$ (homoscedastic model) and $\sigma_Y(x,\theta)=a+\sigma b_Y(x,\theta)$
(heterogeneous model), and $b_Y$ is usually chosen to be a linear function.
Thus, even though the latent process $X$ is described as the solution of the $SDE$ (\ref{eq:state_space_sde_pkpd}),
the model for the (discretely) observed data is postulated to be arising from independent normal distributions, conditional
on the discretized version of the diffusion process $X$. This simplifies inference proceedings to a large extent,
particularly when the Markov transition model associated with (\ref{eq:state_space_sde_pkpd}) is available explicitly.
Here we recall that under suitable regularity conditions, the solution of (\ref{eq:state_space_sde_pkpd})
is a continuous time Markov process (see, for example, \ctn{Arnold74}, \ctn{Oksendal03}, \ctn{Mao11}). 
If the Markov transition model is not available in closed form, then various approximations are proposed
in the literature to approximate the likelihood of $\theta$, using which the $MLE$ of $\theta$
or the posterior distribution of $\theta$ is obtained. Under special cases, for instance, when
$\sigma_Y(x,\theta)=\sigma$, $b_Y(x,\theta)=b_{\theta}x$, $\sigma_X(x,\theta)=\sigma_{\theta}$, 
$b_X(x_t,t,\theta)=a_{\theta}x_t+c_{\theta}(t)$, an explicit form of the likelihood (based on discretization) is available,
and the resulting $MLE$ has been shown to be consistent and asymptotically normal by \ctn{Favetto10}, but
in more general, non-linear situations, theoretical results do not seem to be available.
A comprehensive account of the methods of approximating the $MLE$ and posterior distribution of $\theta$, with
discussion of related computational issues and theoretical results, have been provided in \ctn{Donnet13}.

Our interest in this article is primarily the investigation of asymptotic parametric inference, as $T\rightarrow\infty$, 
from both classical and Bayesian
perspectives, in the context of state space models where the models for the observed data as well as the latent process,
are both described by $SDE$s. In Section \ref{subsec:particle_filtering_asymptotics} we show that such asymptotic parametric inference also addresses consistency of the 
so-called particle filtering problem associated with the joint posterior distribution of the parameters and the latent states $X(t)$ given the 
data $\left\{Y(s):0\leq s\leq t\right\}$. For relatively recent research works on the particle filtering problem in non-$SDE$ setups, see, for example,
\ctn{Chopin13}, \ctn{Crisan13}, \ctn{Urteaga16}, \ctn{Martino17}.

In our knowledge, asymptotic inference in such models 
has not been hitherto investigated.
In our proceedings we assume a somewhat generalized version of the state space $SDE$s  
described by (\ref{eq:data_sde0}) and (\ref{eq:state_space_sde0}) 
in that the drift function $b_Y$ depends upon $Y(t)$, in addition to $X(t)$ and $t$; moreover, we assume that
there is a diffusion coefficient $\sigma_Y(Y(t),X(t),t)$ associated with the Wiener process $W_Y(t)$
that drives the observational $SDE$ (\ref{eq:data_sde0}); a practical instance of such a state space model
in the case of bacterial growth can be found in \ctn{Moller12}.
We further assume that there is a common set of parameters $\theta$ associated with both the $SDE$s, which are
of interest. In particular, we assume that there exist appropriate real-valued, known,
functions for $\theta$, $\psi_Y(\theta)$ and
$\psi_X(\theta)$, such that the drift functions are $\psi_Y(\theta)b_Y(Y(t),X(t),t)$ and $\psi_X(\theta)b_X(X(t),t)$, 
respectively. In Section \ref{sec:assumptions} we clarify that $\psi_Y(\theta)$ and $\psi_X(\theta)$ offers
very general scope of parameterizations by mapping the perhaps high-dimensional (although finite-dimensional)
quantity $\theta$ to appropriate real-valued functional forms composed of the elements of $\theta$. 
We also assume that the
diffusion coefficients of the respective $SDE$s are independent of $\theta$.
A key assumption in our approach to asymptotic investigation is that $X$ is stochastically stable. In a nutshell,
in this article, by stochastic stability of $X$ we mean that 
\begin{equation}
|x(t)|\leq \xi\lambda(t)~~\mbox{for all}~t\geq 0,
\label{eq:ss1}
\end{equation}
almost surely, for all initial values $x(0)\in\mathbb R$,
where $\lambda(t)\rightarrow 0$ as $t\rightarrow\infty$, 
and $\xi$ is a non-negative, finite random
variable depending upon $x(0)$.
For comprehensive details regarding various versions of stochastic stability of solutions of $SDE$s,
see \ctn{Mao11}. 

It is to be noted that our model clearly corresponds to a dependent setup, and establishment of asymptotic
results are therefore can not be achieved by the state-of-the-art methods that typically deal with
at least independent situations. For Bayesian asymptotics we find the consistency results of \ctn{Shalizi09} 
and the general result on posterior asymptotic normality of \ctn{Schervish95} useful
for our purpose, while for classical asymptotics we obtain a suitable asymptotic approximation to the target
log-likelihood, which helped us establish strong consistency, as well as asymptotic normality of the $MLE$.

Once we establish classical and Bayesian asymptotic results associated with our state space $SDE$ model, we then 
extend our model to random effects state space model (see \ctn{Maud12}, for instance, for $SDE$ based
random effects model), where we model each time series data available on $n$ individuals
using our state space model, assuming that the effects $\psi_{Y_i}(\theta)$ and $\psi_{X_i}(\theta)$ for individual $i$ 
are parameterized by $\theta$, which is the parameter of interest. From the classical point of view, this is not a 
random effects model technically since $\theta$ is treated as a fixed quantity, but from the Bayesian viewpoint, 
a prior on $\theta$
renders the effects random. Slightly abusing terminology for the sake of convenience, we continue to call the model 
random effects stochastic $SDE$, from both classical and Bayesian perspectives.  
Under such random effects $SDE$ model we seek asymptotic classical and Bayesian inference on $\theta$ as both number
of individuals, $n$, and the domain of observations $[0,T_i];~i=1,\ldots,n$ increase indefinitely. For our
purpose we assume $T_i=T$ for each $i$. Here we remark that \ctn{Donnet13} discuss population $SDE$ models
with measurement errors; see also \ctn{Overgaard05}, \ctn{Donnet08}, \ctn{Yan14}, \ctn{Leander15}; 
for the $i$-th individual 
such models are of the same form as  
(\ref{eq:data_sde_pkpd}) and (\ref{eq:state_space_sde_pkpd}), but specifics depending upon $i$, and with
$\theta$ replaced with $\phi_i$, where 
$\left\{\phi_1,\ldots,\phi_n\right\}$ are independently and identically distributed
with some distribution with parameter $\theta$, say, which is one of the parameters of interest.
This is a genuine random effects model unlike ours, but here only the latent process $X$ is based upon
$SDE$. Theoretical results do not exist for this setup; see \ctn{Donnet13}. On the other hand,
even though our random effects state space $SDE$ model is completely based upon $SDE$s, the simplified
form of the effects, parameterized by a common $\theta$, enables us to obtain desired asymptotic results
for both classical and Bayesian paradigms. Indeed, in our case it is certainly possible to postulate a genuine
random effects state space $SDE$ model by replacing $\psi_{Y_i}(\theta)$ and $\psi_{X_i}(\theta)$ with 
$iid$ random effects $\phi_{Y_i}$ and $\phi_{X_i}$, having distributions parameterized by quantities of
inferential interest $\theta_Y$ and $\theta_X$, say, but in this setup complications arise regarding handling the
observed integrated likelihood and its associated bounds, which does not assist in our asymptotic investigations.

Discretization of our state space $SDE$ models is essential for practical applications such as in fields
of pharmacokinetics/pharmacodynamics, where continuous time data are usually unavailable. We show in the supplement that
the same asymptotic results go through in discretized situations.

In our proceedings with each setup, we first investigate Bayesian consistency, then consistency and
asymptotic normality of the $MLE$, and finally asymptotic posterior normality. 
One reason behind this sequence is that the proofs of the results on posterior normality depend upon
the proofs of the results of consistency and asymptotic normality of $MLE$, which, in turn, depend upon the 
proofs associated with Bayesian posterior consistency. Moreover, adhering to this sequence allows us to 
introduce the assumptions in a sequential manner, so that an overall logical order could be maintained
throughout the paper.

The rest of our article is organized as follows. In Section \ref{sec:ss_formulations} we introduce our state
space $SDE$ model and provide an overview of the asymptotic results in Section \ref{sec:overview_main_results}. We list the various sets
of assumptions including stochastic stability of the 
solution of the latent $SDE$, in Section \ref{sec:assumptions}. 
Development of the asymptotic theory requires asymptotic approximation of the true and observed likelihoods. Such asymptotic
approximations are developed in Section \ref{sec:asymp_approx_likelihoods}, under suitable regularity conditions.
Next, in Section \ref{sec:posterior_convergence}, with further regularity conditions,
we prove posterior convergence of $\theta$ by proving validity of the conditions of Shalizi stated formally for our state space $SDE$ setup
in Section \ref{sec:assumptions_shalizi} of the supplement. 
We prove strong consistency and asymptotic normality of the $MLE$ in Section \ref{sec:classical_asymptotics}, under further extra assumptions.
With a few more regularity conditions, In Section \ref{sec:posterior_normal} we establish asymptotic posterior normality of $\theta$. 
We introduce random effects state space $SDE$ models in Section \ref{sec:random_effects_brief} and provide a briefing of the asymptotic results,
with the details in Section \ref{sec:random_effects} of the supplement.
Finally, in Section \ref{sec:conclusion} we provide a brief summary of our work, discuss some key issues, and identify
future research agenda. The extension of our theory for state space $SDE$ models
with multidimensional linear random effects and in the case of discretized data are discussed, respectively, 
in Sections \ref{sec:multidimensional} and \ref{sec:discrete_data} of the supplement.


\section{State space $SDE$}
\label{sec:ss_formulations}

\subsection{True and postulated state space $SDE$ models}
\label{subsec:formulation}
First consider the following ``true" state space $SDE$:
\begin{align}
d Y(t)&=\phi_{Y,0} b_Y(Y(t),X(t),t)dt+\sigma_Y(Y(t),X(t),t)dW_Y(t);
\label{eq:data_sde1}\\
d X(t)&=\phi_{X,0}b_X(X(t),t)dt+\sigma_X(X(t),t)dW_X(t),
\label{eq:state_space_sde1}
\end{align}
for $t\in [0,b_T]$, where $b_T\rightarrow\infty$, as $T\rightarrow\infty$. 
The first $SDE$, namely, (\ref{eq:data_sde1}) is the true observational
$SDE$ and is associated with the observed data. The second $SDE$ (\ref{eq:state_space_sde1})
is the true evolutionary, unobservable $SDE$.
In the above two equations, we assume that $\phi_{Y,0}$ and $\phi_{X,0}$ are both explained
by a ``true" set of parameters $\theta_0$, through known but perhaps different functions of $\theta_0$.
In other words, we assume that $\phi_{Y,0}=\psi_Y(\theta_0)$ and $\phi_{X,0}=\psi_X(\theta_0)$, where
$\psi_Y$ and $\psi_X$ are known functions. 
Note that this is a general formulation, where we allow the possibility $\theta_0=(\theta_{Y,0},\theta_{X,0})$
and choice of $\psi_Y$ and $\psi_X$ such that 
$\psi_Y(\theta_0)=\theta_{Y,0}$ and $\psi_X(\theta_0)=\theta_{X,0}$, 
for scalars $\theta_{Y,0}$ and $\theta_{X,0}$.  
In this instance, the observational and evolutionary $SDE$s have their
own sets of parameters. We also allow common subsets of the parameter vector $\theta_0$ to feature in
the two $SDE$s. For instance, $\psi_Y(\theta_0)=\theta_{Y,0}+\theta_{X,0}$ and 
$\psi_X(\theta_0)=\theta_{X,0}$. Indeed, $\theta_0$ can be any finite-dimensional vector, appropriately
mapped to the real line by $\psi_Y$ and $\psi_X$.
We wish to learn about the set of parameters $\theta_0$, which would enable learning 
about $\phi_{Y,0}$ and $\phi_{X,0}$ simultaneously. For our purpose, we assume that $(\psi_Y(\theta),\psi_X(\theta))$
is identifiable in $\theta$, that is, $(\psi_Y(\theta_1),\psi_X(\theta_1))=(\psi_Y(\theta_2),\psi_X(\theta_2))$
implies $\theta_1=\theta_2$.

Our modeled state space $SDE$ is analogously given, for $t\in [0,b_T]$ by:
\begin{align}
d Y(t)&=\phi_Y b_Y(Y(t),X(t),t)dt+\sigma_Y(Y(t),X(t),t)dW_Y(t);
\label{eq:data_sde2}\\
d X(t)&=\phi_Xb_X(X(t),t)dt+\sigma_X(X(t),t)dW_X(t),
\label{eq:state_space_sde2}
\end{align}
where $\phi_{Y}=\psi_Y(\theta)$ and $\phi_{X}=\psi_X(\theta)$.

Throughout, we assume that the initial values associated with the $SDE$s (\ref{eq:data_sde1}),
(\ref{eq:state_space_sde1}), (\ref{eq:data_sde2}) and (\ref{eq:state_space_sde2}), are non-random.
It is worth mentioning in this context that for stochastic stability it is enough to
assume non-randomness of the initial value; see \ctn{Mao11}, page 111, for a proof of this.

We wish to establish consistency and asymptotic normality of 
the maximum likelihood estimator ($MLE$) and the posterior distribution of 
$\theta$, as $T\rightarrow\infty$.
For technical reasons we shall consider the likelihood for $t\in [a_T,b_T]$, where $a_T\rightarrow\infty$
and $(b_T-a_T)\rightarrow\infty$, as $T\rightarrow\infty$.
In particular, we assume that $(b_T-a_T)\geq T$.

\subsection{Connection of parametric asymptotic inference with the asymptotics of the particle filtering problem}
\label{subsec:particle_filtering_asymptotics}

As already mentioned, in this article we focus on classical and Bayesian asymptotic inference on the parameter $\theta$. However, such asymptotic parametric inference
automatically leads to asymptotic inference regarding the particle filtering problem. 
To clarify, first let 
$\mathcal Y_t=\left\{Y(s):0\leq s\leq t\right\}$, for $t\in [0,b_T]$, and let 
$\hat\theta_T$ denote the $MLE$ of $\theta$ or the posterior expectation of $\theta$, given the data
$\mathcal Y_T$. Then provided that $\hat\theta_T\rightarrow\theta_0$ almost surely (or in probability), for each $t\in[0,b_T]$, the posterior distribution
$\pi\left(X(t)|\hat\theta_T,\mathcal Y_t\right)\rightarrow \pi\left(X(t)|\theta_0,\mathcal Y_t\right)$, as $T\rightarrow\infty$, 
almost surely (or in probability), if $\pi\left(X(t)|\theta,\mathcal Y_t\right)$ is continuous in $\theta$. As a simple example, let us assume that
$b_Y(Y(t),X(t),t)=L(t)X(t)$, $b_X(X(t),t)=H(t)X(t)$, $\sigma_Y(Y(t),X(t),t)\equiv 1$ and $\sigma_X(X(t),t)=\sigma_X(t)$. Also, let us assume that
$\psi_Y(\theta)$ and $\psi_X(\theta)$ are continuous in $\theta$. Then the Kalman-Bucy filter ensures that $\pi\left(X(t)|\theta,\mathcal Y_t\right)$
is a Gaussian density with mean and variances depending upon $t$, and the density is continuous in $\theta$. Letting $\mathcal X_t=\left\{X(s):0\leq s\leq b_T\right\}$,
we similarly have 
$\pi\left(\mathcal X_t|\hat\theta_T,\mathcal Y_t\right)\rightarrow \pi\left(\mathcal X_t|\theta_0,\mathcal Y_t\right)$, as $T\rightarrow\infty$,
almost surely (or in probability).


\section{A brief overview of the main asymptotic results}
\label{sec:overview_main_results}

\subsection{Posterior convergence of $\theta$}
\label{subsec:posterior_consistency}

Our main result on posterior convergence of $\theta$ is based on verification of a general posterior convergence result of \ctn{Shalizi09}, which amounts
to validating seven regularity conditions required by Shalizi's result, which we denote by (A1) -- (A7). 
We present the assumptions and the result of Shalizi in Section \ref{sec:assumptions_shalizi}
of the supplement. The most essential notions, the key assumption, and our main result on posterior convergence with a brief sketch of the proof utilizing the key
assumption of Shalizi, are presented below.

Let $\mathcal F_T=\sigma(\left\{Y(s):s\in[a_T,b_T]\right\})$ 
denote the $\sigma$-algebra generated by $\left\{Y(s):s\in[a_T,b_T]\right\}$.
Let $\mathcal T$ denote the $\sigma$-algebra associated with the $d~(\geq 1)$-dimensional parameter space $\Theta$.

Let $p_T(\theta_0)$ denote the marginal likelihood of $\{Y(t):t\in [a_T,b_T]\}$ of the true model (\ref{eq:data_sde1}) and (\ref{eq:state_space_sde1}).
Also, let $L_T(\theta)$ be the modeled likelihood of $\{Y(t):t\in [a_T,b_T]\}$ of the postulated model (\ref{eq:data_sde2}) and (\ref{eq:state_space_sde2}).
We denote $\frac{L_T(\theta)}{p_T(\theta_0)}$ by $R_T(\theta)$.
For every $\theta\in\Theta$, the Kullback-Leibler divergence rate is given by
\begin{equation*}
h(\theta)=\underset{T\rightarrow\infty}{\lim}~\frac{1}{b_T-a_T}E_{\theta_0}\left(-\log R_T(\theta)\right),
\end{equation*}
where $E_{\theta_0}$ denotes the expectation is with respect to the true model. 

For $A\subseteq\Theta$, let
\begin{align}
h\left(A\right)&=\underset{\theta\in A}{\mbox{ess~inf}}~h(\theta);\notag\\
J(\theta)&=h(\theta)-h(\Theta);\notag\\
J(A)&=\underset{\theta\in A}{\mbox{ess~inf}}~J(\theta).\notag
\end{align}
The above essential infimums are with respect to the prior $\pi$ assigned for $\theta$.


With the above notions, our posterior convergence results are summarized by Theorem 
\ref{theorem:bayesian_convergence_brief}. 

\begin{theorem}
\label{theorem:bayesian_convergence_brief}
Assume that the data was generated by the true model given by (\ref{eq:data_sde1}) and (\ref{eq:state_space_sde1}), but
modeled by (\ref{eq:data_sde2}) and (\ref{eq:state_space_sde2}). 
For the prior $\pi$ on $\theta$, consider any set $A\in\mathcal T$ with $\pi(A)>0$ and $h(A)>h(\Theta)$. 
Then, under suitable assumptions, almost surely,
\begin{equation*}
\underset{T\rightarrow\infty}{\lim}~\pi(A|\mathcal F_T)=0.
\end{equation*}
Moreover, if the set $A$ satisfies another technical condition, 
then almost surely,
\begin{equation*}
\underset{T\rightarrow\infty}{\lim}~\frac{1}{b_T-a_T}\log\pi(A|\mathcal F_T)=-J(A).
\end{equation*}
\end{theorem}
\begin{proof}[Sketch of the proof]
The proof follows by verifying the seven assumptions of Shalizi, which are shown to hold under appropriate conditions.
The most important result guiding posterior convergence is the asymptotic equipartition property, which is given in this case by
\begin{equation*}
\frac{1}{b_T-a_T}\log R_T(\theta)\rightarrow -\frac{1}{2}\left[K_Y\left(\phi_Y-\phi_{Y,0}\right)^2+
K_X\left(\phi_X-\phi_{X,0}\right)^2+K_X\left(\phi^2_{X,0}-\phi^2_{X}\right)\right]=-h(\theta),
\end{equation*}
where
\begin{align*}
h(\theta)&=\frac{1}{2}\left[K_Y\left(\phi_Y-\phi_{Y,0}\right)^2+K_X\left(\phi_X-\phi_{X,0}\right)^2
+K_X\left(\phi^2_{X,0}-\phi^2_{X}\right)\right]\notag\\
&=\frac{1}{2}\left[K_Y\left(\psi_Y(\theta)-\psi_Y(\theta_0)\right)^2
+K_X\left(\psi_X(\theta)-\psi_X(\theta_0)\right)^2
+K_X\left(\psi^2_X(\theta_0)-\psi^2_{X}(\theta)\right)\right].
\end{align*}
In the above, $K_X~(>0)$ and $K_Y~(>0)$ are the limits of the bounds of $b^2_Y(y,x,t)/\sigma^2_Y(y,x,t)$ and $b^2_X(x,t)/\sigma^2_X(x,t)$, respectively, as $T\rightarrow\infty$.

This result is achieved using the following approximations proved subsequently: $p_T(\theta_0)\stackrel{a.s.}{\sim}\hat p_T(\theta_0)$ and 
$L_T(\theta)\stackrel{a.s.}{\sim}\hat L_T(\theta)$, where
\begin{equation*}
\hat p_T(\theta_0)
=\exp\left(\frac{(b_T-a_T)K_Y\phi^2_{Y,0}}{2}+\phi_{Y,0}\sqrt{K_Y}\left(W_Y(b_T)-W_Y(a_T)\right)
+(b_T-a_T)K_X\phi^2_{X,0}\right).
\end{equation*}
\begin{align}
\hat L_T(\theta)
&=\exp\left((b_T-a_T)K_Y\phi_Y\phi_{Y,0}+\phi_Y\sqrt{K_Y}\left(W_Y(b_T)-W_Y(a_T)\right)\right.\notag\\
&\qquad\qquad\left.-\frac{(b_T-a_T)K_Y\phi^2_Y}{2}+(b_T-a_T)K_X\phi_X\phi_{X,0}\right),\notag
\end{align}
and then noting that, as $T\rightarrow\infty$,
\begin{align}
&\frac{1}{b_T-a_T}\log R_T(\theta)=\frac{1}{b_T-a_T}\log\left(\frac{L_T(\theta)}{p_T(\theta_0)}\right)\notag\\
&\stackrel{a.s.}{\sim}
-\frac{K_Y}{2}\left(\phi_Y-\phi_{Y,0}\right)^2
+\sqrt{K_Y}\left(\phi_Y-\phi_{Y,0}\right)\frac{\left(W_Y(b_T)-W_Y(a_T)\right)}{b_T-a_T}\notag\\
&\qquad\qquad-\frac{K_X}{2}\left(\phi_X-\phi_{X,0}\right)^2
+\frac{K_X}{2}\left(\phi^2_X-\phi^2_{X,0}\right)\notag\\
&\stackrel{a.s.}{\longrightarrow} 
-\frac{1}{2}\left[K_Y\left(\phi_Y-\phi_{Y,0}\right)^2+
K_X\left(\phi_X-\phi_{X,0}\right)^2+K_X\left(\phi^2_{X,0}-\phi^2_{X}\right)\right].\notag
\end{align}

It is important to note that compactness of the parameter space $\Theta$ is not necessary for Theorem \ref{theorem:bayesian_convergence_brief} to hold. Instead,
we constructed appropriate ``sieves" of the form $\mathcal G_T=\left\{\theta:\left|\psi_Y(\theta)\right|\leq\exp\left(\beta\left(b_T-a_T\right)\right)\right\}$
with $\beta>2h\left(\Theta\right)$
that are compact for each $T$ and increasing in $T$ and such the prior probability of the complement $\mathcal G^c_T$
is exponentially small, and satisfies some other technical conditions that essentially guarantee posterior convergence, along with the asymptotic equipartition property. 
\end{proof}

\begin{remark}
In particular, let $A_{\epsilon}=\left\{\theta\in\Theta:h(\theta)>h\left(\Theta\right)+\epsilon\right\}$, for $\epsilon>0$. Then note that 
$h\left(A_{\epsilon}\right)>h\left(\Theta\right)$, for any $\epsilon>0$. Let $\pi\left(A_{\epsilon}\right)>0$. Then by the first part of 
Theorem \ref{theorem:bayesian_convergence_brief}, 
$\pi\left(A_{\epsilon}|\mathcal F_T\right)\rightarrow 0$, almost surely, as $T\rightarrow\infty$, for any $\epsilon>0$.
It is also important to note that if $\theta_0$ belongs to the support of the prior on $\Theta$, then $h\left(\Theta\right)=0$. In this case,
the posterior probability of $A_{\epsilon}=\left\{\theta\in\Theta:h(\theta)>\epsilon\right\}$ tends to zero almost surely, for any $\epsilon>0$. 
\end{remark}


\subsection{Consistency of the $MLE$ of $\theta$}
\label{subsec:consistency_MLE}
Let $\theta\in\Theta\subseteq\mathbb R^d$, where $\Theta$ is the $d~(\geq 1)$-dimensional, compact parameter space.
Our main result on consistency of the $MLE$ of $\theta$ can be formalized as the following theorem.
\begin{theorem}
\label{theorem:mle_strong_consistency1_brief}
Assume that the data was generated by the true model given by (\ref{eq:data_sde1}) and (\ref{eq:state_space_sde1}), but
modeled by (\ref{eq:data_sde2}) and (\ref{eq:state_space_sde2}).
Then under appropriate regularity conditions the $MLE$ $\hat\theta_T$ of $\theta$ is strongly consistent in the sense that
$\hat\theta_T\stackrel{a.s.}{\longrightarrow}\theta_0$. 
\end{theorem}
\begin{proof}[Sketch of the proof]
Identifiability of the model and uniqueness of the $MLE$ follow from our assumptions.
To prove strong consistency of the $MLE$, we first note that the $MLE$ can be approximated by maximizing the function
$$\tilde g_T(\theta)=g_{Y,T}(\theta)+g_{X,T}(\theta)$$ with respect to $\theta$, where
\begin{align}
g_{Y,T}(\theta)&=-\frac{K_Y}{2}\left(\psi_Y(\theta)-\psi_Y(\theta_0)\right)^2
+\sqrt{K_Y}\left(\psi_Y(\theta)-\psi_Y(\theta_0)\right)\frac{W_Y(b_T)-W_Y(a_T)}{b_T-a_T};\notag\\
g_{X,T}(\theta)&=-\frac{K_X}{2}\left(\psi_X(\theta)-\psi_X(\theta_0)\right)^2
+\frac{K_X}{2}\left(\psi^2_X(\theta)-\psi^2_X(\theta_0)\right).\notag
\end{align}

Letting $\hat\theta_{T}$ denote the $MLE$, note that
\begin{equation*}
0=\tilde g'_T(\hat\theta_{T})=\tilde g'_T(\theta_0)+\tilde g''_T(\theta^*_T)(\hat\theta_T-\theta_0),
\end{equation*}
where $\theta^*_T$ lies between $\theta_0$ and $\hat\theta_T$. Since $\tilde g'_T(\theta_0)\stackrel{a.s.}{\longrightarrow}0$ as $T\rightarrow\infty$
and since $\tilde g''_T(\theta^*_T)$ is positive definite for $T\geq 1$ under appropriate assumptions, it holds that
$\hat\theta_T\stackrel{a.s.}{\longrightarrow}\theta_0$, as $T\rightarrow\infty$.
\end{proof}
\begin{remark}
Note that compactness of $\Theta$ is not necessary for Bayesian consistency, in contrast with consistency of the $MLE$.
\end{remark}

\subsection{Asymptotic normality of the $MLE$ of $\theta$}
\label{subsec:normality_MLE}

For asymptotic normality of the $MLE$ of $\theta$, the result is summarized by the following theorem.
\begin{theorem}
\label{theorem:mle_normality1_brief}
Assume that the data was generated by the true model given by (\ref{eq:data_sde1}) and (\ref{eq:state_space_sde1}), but
modeled by (\ref{eq:data_sde2}) and (\ref{eq:state_space_sde2}).
Then under suitable assumptions the $MLE$ of $\theta$ is asymptotically normal in the sense that
$\sqrt{b_T-a_T}\left(\hat\theta_T-\theta_0\right)
\stackrel{\mathcal L}{\longrightarrow}N_d\left(0,\mathcal I^{-1}(\theta_0)\right)$. 
Here $\mathcal I(\theta)$ is the matrix with $(j,k)$-th element given by
\begin{equation*}
\left\{\mathcal I(\theta)\right\}_{jk}=K_Y\left[\frac{\partial\psi_Y(\theta)}{\partial\theta_j}\frac{\partial\psi_Y(\theta)}
{\partial\theta_k}\right].
\end{equation*}

\end{theorem}
\begin{proof}[Sketch of the proof]
Asymptotic normality follows easily from the above developments on consistency of $MLE$, and the 
fact that $\theta^*_T\stackrel{a.s.}{\longrightarrow}\theta_0$, and $\frac{W_Y(b_T)-W_Y(a_T)}{b_T-a_T}\stackrel{a.s.}{\longrightarrow} 0$, 
as $T\rightarrow\infty$. 

Observe that $\left\{\mathcal I(\theta_0)\right\}_{jk}$ is the covariance between the $j$-th
and the $k$-th components of $\sqrt{b_T-a_T}\tilde g'_T(\theta_0)$, and so $\mathcal I(\theta_0)$ is non-negative definite.
\end{proof}

\subsection{Asymptotic posterior normality of $\theta$}
\label{subsec:posterior_normality}

We prove posterior normality of $\theta$ by verifying the seven regularity conditions of Theorem 7.102 of \ctn{Schervish95}.
\begin{theorem}
\label{theorem:theorem5_brief}
Assume that the data was generated by the true model given by (\ref{eq:data_sde1}) and (\ref{eq:state_space_sde1}), but
modeled by (\ref{eq:data_sde2}) and (\ref{eq:state_space_sde2}).
Then denoting 
$\Psi_T=(b_T-a_T)^{\frac{1}{2}}\mathcal I^{\frac{1}{2}}(\theta_0)\left(\theta-\hat\theta_T\right)$,
for each compact subset 
$B$ of $\mathbb R^d$ and each $\epsilon>0$, the following holds under appropriate assumptions:
\begin{equation*}
\lim_{T\rightarrow\infty}P_{\theta_0}
\left(\sup_{\Psi_T\in B}\left\vert\pi(\Psi_T\vert\mathcal F_T)-\varrho(\Psi_T)\right\vert>\epsilon\right)=0,
\end{equation*}
where $\varrho(\cdot)$ denotes the density of the standard normal distribution.
\end{theorem}
\begin{proof}[Sketch of the proof]
Here we assume that $\Theta$ is compact 
which enables us to uniformly approximate
$\frac{1}{b_T-a_T}\log R_T(\theta)$ by $g_{Y,T}(\theta)+g_{X,T}(\theta)$ for $\theta\in\Theta$. 
Hence, $\frac{1}{b_T-a_T}\log L_T(\theta)$ can be uniformly approximated by 
$\frac{1}{b_T-a_T}\tilde\ell_T(\theta)=g_{Y,T}(\theta)+g_{X,T}(\theta)+\frac{1}{b_T-a_T}\log p_T(\theta_0)$, for $\theta\in\Theta$.
This is the key idea, and by working with the first three differentials of $\tilde\ell_T(\theta)$, in conjunction with Taylor's series expansion and our proven result
that $\hat\theta_T\stackrel{a.s.}{\longrightarrow}\theta_0$, the seven regularity conditions of Theorem 7.102 of \ctn{Schervish95} are relatively straightforward to verify.
\end{proof}

\section{Regularity conditions}
\label{sec:assumptions}
\subsection{Assumptions regarding $b_Y$ and $\sigma_Y$}
\label{subsec:assumptions_Y}
\begin{itemize}
\item[(H1)] For every $T>0$, and integer $\eta\geq 1$, given any $x$, there exists a positive constant
$K_{Y,x,T,\eta}$ such that for all $t\in [0,b_T]$ and all $(y_1,y_2)$ with $\max\{y_1,y_2\}\leq\eta$,
$$\max\left\{\left[b_Y(y_1,x,t)-b_Y(y_2,x,t)\right]^2,\left[\sigma_Y(y_1,x,t)-\sigma_Y(y_2,x,t)\right]^2\right\}
\leq K_{Y,x,T,\eta}|y_1-y_2|^2.$$
\item[(H2)] For every $T>0$, given any $x$, there exists a positive constant $K_{x,T}$ such that for all 
$(y,t)\in\mathbb R\times [0,T]$, 
$$\max\left\{b^2_Y(y,x,t), \sigma^2_Y(y,x,t)\right\}\leq K_{x,T}\left(1+y^2\right).$$
\item[(H3)]  For every $T>0$, there exist positive constants 
$K_{Y,1,T}$, $K_{Y,2,T}$, $\alpha_{Y,1}$, $\alpha_{Y,2}$
such that for all $(x,t)\in\mathbb R\times [0,b_T]$,
$$K_{Y,1,T}\left(1-\alpha_{Y,1}x^2\right)
\leq\frac{b^2_Y(y,x,t)}{\sigma^2_Y(y,x,t)}
\leq K_{Y,2,T}\left(1+\alpha_{Y,2}x^2\right),$$
where $K_{Y,1,T}\rightarrow K_Y$ and $K_{Y,2,T}\rightarrow K_Y$ as $T\rightarrow\infty$; $K_Y$ being a 
positive constant.
We further assume that for $j=1,2$, $(b_T-a_T)\left|K_{Y,j,T}-K_Y\right|\rightarrow 0$, as $T\rightarrow\infty$. 
\end{itemize}
In (H3) we have assumed that the bounds of $\frac{b^2_Y(y,x,t)}{\sigma^2_Y(y,x,t)}$ do not depend upon $y$,
which is somewhat restrictive. Dependence of the bounds on $y$ can be insisted upon, but at the cost of the assumption
of stochastic stability of $Y$ in addition to that of $X$. See Section \ref{sec:conclusion} for details regarding
the modified assumption. All our results remain intact under the modified assumption.
It is also important to clarify that the lower bound in (H3), when utilized in our $SDE$ context,
becomes non-negative after possibly a few time steps, thanks to the stochastic stability assumption
which ensures (\ref{eq:ss1}).

\subsection{Assumptions regarding $b_X$ and $\sigma_X$}
\label{subsec:assumptions_X}
\begin{itemize}
\item[(H4)] $b_X(0,t)=0=\sigma_X(0,t)$ for all $t\geq 0$.
\item[(H5)] For every $T>0$, and integer $\eta\geq 1$, there exists a positive constant
$K_{T,\eta}$ such that for all $t\in [0,b_T]$ and all $(x_1,x_2)$ with $\max\{x_1,x_2\}\leq\eta$,
$$\max\left\{\left[b_X(x_1,t)-b_X(x_2,t)\right]^2,\left[\sigma_X(x_1,t)-\sigma_X(x_2,t)\right]^2\right\}
\leq K_{T,\eta}|x_1-x_2|^2.$$
\item[(H6)] For every $T>0$, there exists a positive constant $K_T$ such that for all $(x,t)\in\mathbb R\times [0,b_T]$,
$$\max\left\{b^2_X(x,t),\sigma^2_X(x,t)\right\}\leq K_T\left(1+x^2\right).$$
\item[(H7)]  For every $T>0$, there exist positive constants 
$K_{X,1,T}$, $K_{X,2,T}$, $\alpha_{X,1}$, $\alpha_{X,2}$
such that for all $(x,t)\in\mathbb R\times [0,b_T]$,
$$K_{X,1,T}\left(1-\alpha_{X,1}x^2\right)
\leq\frac{b^2_X(x,t)}{\sigma^2_X(x,t)}
\leq K_{X,2,T}\left(1+\alpha_{X,2}x^2\right),$$
where $K_{X,1,T}\rightarrow K_X$ and $K_{X,2,T}\rightarrow K_X$, as $T\rightarrow\infty$;
$K_X$ being a positive constant. 
We also assume that for $j=1,2$, $(b_T-a_T)\left|K_{X,j,T}-K_X\right|\rightarrow 0$, as $T\rightarrow\infty$. 
\end{itemize}

\subsection{Further assumptions ensuring almost sure stochastic stability of $X(t)$}
\label{subsec:assumptions_stochastic_stability}
Let $\mathcal C$ denote the family of
all continuous non-decreasing functions $f:\mathbb R^+\mapsto\mathbb R^+$ such that $f(0)=0$ and
$f(r)>0$ when $r>0$.

Let $S_h=\{x\in\mathbb R:|x|<h\}$ and $\mathbb C(S_h\times[0,\infty);\mathbb R^+)$ denote the family of all continuous functions 
$V(x,t)$ from $S_h\times[0,\infty)$ to $\mathbb R^+$ with continuous first partial derivatives with respect
to $x$ and $t$. Also, let $\mathfrak C(S_h\times[0,\infty);\mathbb R^+)$, where $0<h\leq\infty$, denote the
family of non-negative functions $V(x,t)$ defined on
$S_h\times\mathbb R^+$ such that they are continuously twice differentiable in $x$ and once in $t$.
Let 
\begin{equation*}
LV(x,t) = V_t(x,t)+V_x(x,t)b_X(x,t)+\frac{1}{2}\sigma^2_X(x,t)V_{xx}(x,t),
\end{equation*}
where $V_t=\frac{\partial V}{\partial t}$, $V_x=\frac{\partial V}{\partial x}$,
and $V_{xx}=\frac{\partial^2 V}{\partial x^2}$.

With these definitions and notations, we now make the following assumption:
\begin{itemize}
\item[(H8)] Let $p>0$ and let there exist a function $V\in\mathfrak C(S_h\times[0,\infty);\mathbb R^+)$,
a continuous non-decreasing function $\gamma:\mathbb R^+\mapsto\mathbb R^+$ such that
$\gamma(t)\rightarrow\infty$ as $t\rightarrow\infty$, and a continuous function $\breve{\eta}:\mathbb R^+\mapsto\mathbb R^+$
such that $\int_0^{\infty}\breve{\eta}(t)<\infty$. Assume that
for $x\neq 0$, $t\geq 0$,
\[
\gamma(t) |x|^p\leq V(x,t)~\mbox{and}~LV(x,t)\leq \breve{\eta}(t).
\]
\end{itemize}
Thanks to Theorem 6.2 of \ctn{Mao11} (page 145), assumption (H8) ensures that stochastic stability 
of $X$ of the form $|x(t)|\leq \xi\lambda(t)~~\mbox{for all}~t\geq 0$ holds 
almost surely, for all initial values $x(0)\in\mathbb R$ with 
\begin{equation*}
\lambda(t)=\left[\gamma(t)\right]^{-\frac{1}{p}},
\end{equation*}
where $\xi$ is a non-negative, finite random
variable depending upon $x(0)$.

\section{Asymptotic approximations of the true and modeled likelihoods of the state space $SDE$s} 
\label{sec:asymp_approx_likelihoods}
Let us define
\begin{align}
v_{Y|X,T}&=\int_{a_T}^{b_T}\frac{b^2_Y(Y(s),X(s),s)}{\sigma^2_Y(Y(s),X(s),s)}ds\label{eq:v_Y}\\
u_{Y|X,T}&=\int_{a_T}^{b_T}\frac{b_Y(Y(s),X(s),s)}{\sigma^2_Y(Y(s),X(s),s)}dY(s)\label{eq:u_Y}\\
v_{X,T}&=\int_{a_T}^{b_T}\frac{b^2_X(X(s),s)}{\sigma^2_X(X(s),s)}ds\label{eq:v_X}\\
u_{X,T}&=\int_{a_T}^{b_T}\frac{b_X(X(s),s)}{\sigma^2_X(X(s),s)}dX(s).\label{eq:u_X}
\end{align}
Due to (H3) and (H7), the following hold:
\begin{align}
K_{Y,\xi,T,1} \leq v_{Y|X,T} \leq K_{Y,\xi,T,2}; \label{eq:v_Y_bounds}\\
K_{X,\xi,T,1} \leq v_{X,T} \leq K_{X,\xi,T,2}, \label{eq:v_X_bounds}
\end{align}
where
\begin{align}
K_{Y,\xi,T,1}&=K_{Y,1,T}\left((b_T-a_T)-\alpha_{Y,1}\xi^2\int_{a_T}^{b_T}\lambda^2(s)ds\right);\label{eq:v_Y_lower_bound}\\
K_{Y,\xi,T,2}&=K_{Y,2,T}\left((b_T-a_T)+\alpha_{Y,2}\xi^2\int_{a_T}^{b_T}\lambda^2(s)ds\right);\label{eq:v_Y_upper_bound}\\
K_{X,\xi,T,1}&=K_{X,1,T}\left((b_T-a_T)-\alpha_{X,1}\xi^2\int_{a_T}^{b_T}\lambda^2(s)ds\right);\label{eq:v_X_lower_bound}\\
K_{X,\xi,T,2}&=K_{X,2,T}\left((b_T-a_T)+\alpha_{X,2}\xi^2\int_{a_T}^{b_T}\lambda^2(s)ds\right).\label{eq:v_X_upper_bound}
\end{align}

To proceed, we shall make use of the following relationships between $u_{Y|X,T}$, $v_{Y|X,T}$ 
and $u_{X,T}$, $v_{X,T}$ under the true state space
$SDE$ model described by (\ref{eq:data_sde1}) and (\ref{eq:state_space_sde1}):
\begin{align}
u_{Y|X,T}&=\phi_{Y,0} v_{Y|X,T}+\int_{a_T}^{b_T}\frac{b_Y(Y(s),X(s),s)}{\sigma_Y(Y(s),X(s),s)}dW_Y(s);
\label{eq:u_v_relation}\\
u_{X,T}&=\phi_{X,0} v_{X,T}+\int_{a_T}^{b_T}\frac{b_X(X(s),s)}{\sigma_X(X(s),s)}dW_X(s).
\label{eq:u_v_relation_X}
\end{align}

Let 
\begin{align}
I_{Y,X,T}&=\int_{a_T}^{b_T}\frac{b_Y(Y(s),X(s),s)}{\sigma_Y(Y(s),X(s),s)}dW_Y(s);\notag\\ 
I_{X,T}&=\int_{a_T}^{b_T}\frac{b_X(X(s),s)}{\sigma_X(X(s),s)}dW_X(s).\notag 
\end{align}
Because of (\ref{eq:v_Y_bounds}), (\ref{eq:v_X_bounds}), (\ref{eq:u_v_relation}) 
and (\ref{eq:u_v_relation_X}) the following hold:
\begin{align}
\phi_{Y,0} K_{Y,\xi,T,1}+I_{Y,X,T}&\leq u_{Y|X,T}\leq\phi_{Y,0} K_{Y,\xi,T,2}+I_{Y,X,T}; \label{eq:u_Y_bounds}\\
\phi_{X,0} K_{X,\xi,T,1}+I_{X,T}&\leq u_{X,T}\leq\phi_{X,0} K_{X,\xi,T,2}+I_{X,T}.\notag 
\end{align}


\subsection{True likelihood and its asymptotic approximation}
\label{subsec:true_model}


First note that 
$\exp\left(\phi_{Y,0} u_{Y|X,T}-\frac{\phi^2_{Y,0}}{2}v_{Y|X,T}\right)$
is the conditional density of $Y$ given $X$, with respect to $Q_{T,Y|X}$, the probability
measure associated with (\ref{eq:data_sde1}) on $[a_T,b_T]$, assuming null drift. Also,
$\exp\left(\phi_{X,0} u_{X,T}-\frac{\phi^2_{X,0}}{2}v_{X,T}\right)$
is the marginal density of $X$ with respect to $Q_{T,X}$, the probability measure 
associated with the latent state $SDE$ (\ref{eq:state_space_sde1}) on $[a_T,b_T]$, 
but assuming null drift.
These are standard results; see for example, \ctn{Lipster01}, \ctn{Oksendal03}, \ctn{Maud12}.

It then follows that the marginal likelihood under the true model (\ref{eq:data_sde1}) and (\ref{eq:state_space_sde1})
is the marginal density of $\left\{Y(t):~t\in[a_T,b_T]\right\}$, given by
\begin{align}
p_T(\theta_0)=\int \exp\left(\phi_{Y,0} u_{Y|X,T}-\frac{\phi^2_{Y,0}}{2}v_{Y|X,T}\right)
\times \exp\left(\phi_{X,0} u_{X,T}-\frac{\phi^2_{X,0}}{2}v_{X,T}\right)dQ_{T,X}
\label{eq:L_T_1}\\
=E_{T,X}\left[\exp\left(\phi_{Y,0} u_{Y|X,T}-\frac{\phi^2_{Y,0}}{2}v_{Y|X,T}\right)
\times \exp\left(\phi_{X,0} u_{X,T}-\frac{\phi^2_{X,0}}{2}v_{X,T}\right)\right],\notag
\end{align}
where $E_{T,X}$ denotes expectation with respect to $Q_{T,X}$.
The following lemma proved in supplement formalizes the dominating measure with respect to which $p_T(\theta_0)$ is
the Radon-Nikodym derivative.
\begin{lemma}
\label{lemma:marginal_likelihood}
The likelihood given by (\ref{eq:L_T_1}) is the density of $\left\{Y(t):~t\in[a_T,b_T]\right\}$ with respect to $Q_{T,Y}$, 
where for any relevant measurable
set $A$, 
\begin{equation*}
Q_{T,Y}(A)=\int_{\mathfrak X_T} dQ_{T,Y|X}(A)dQ_{T,X}=\int_A\int_{\mathfrak X_T} dQ_{T,Y|X}dQ_{T,X}.
\end{equation*}
In the above, $\mathfrak X_T$ stands for the sample space of $\left\{X(t):~t\in[a_T,b_T]\right\}$.  
\end{lemma}
It is important to remark that our likelihood (\ref{eq:L_T_1}) is of a very general form 
and does not usually admit a closed form expression, but this is not at all a requirement for our
asymptotic purpose. Closed form expressions may be necessary when it is of interest to directly maximize
the likelihood with respect to the parameters, and in such cases, more stringent assumptions 
regarding the $SDE$s are necessary. See, for example, \ctn{Fry03}; see also \ctn{Kailath71}. 
Also, observe that our dominating measure $Q_{T,Y}$ is not the Wiener measure, unlike the aforementioned papers,
albeit it reduces to the Wiener measure if $\sigma_Y\equiv 1$ and $\sigma_X\equiv 1$.

\subsubsection{Asymptotic approximation of $p_T(\theta_0)$}
\label{subsubsec:bounds_L_T}

Using (\ref{eq:v_Y_bounds}) and (\ref{eq:u_Y_bounds}) we obtain
\begin{equation*}
B_{L,T}(\theta_0)\leq p_T(\theta_0)\leq B_{U,T}(\theta_0),
\end{equation*}
where
\begin{align}
B_{L,T}(\theta_0)&=E_{T,X}\left(Z_{L,T,\theta_0}(X)\right);\label{eq:B_L_T}\\
B_{U,T}(\theta_0)&=E_{T,X}\left(Z_{U,T,\theta_0}(X)\right),\label{eq:B_U_T}
\end{align}
where
\begin{align}
Z_{L,T,\theta_0}(X)&=\exp\left(\phi^2_{Y,0}K_{Y,\xi,T,1}+\phi_{Y,0}I_{Y,X,T}-\frac{\phi^2_{Y,0}}{2} K_{Y,\xi,T,2}\right)\notag\\
&\quad\quad\times \exp\left(\phi^2_{X,0}K_{X,\xi,T,1}+\phi_{X,0}I_{X,T}-\frac{\phi^2_{X,0}}{2}
K_{X,\xi,T,2}\right)\notag
\end{align}
and
\begin{align}
Z_{U,T,\theta_0}(X)&=\exp\left(\phi^2_{Y,0}K_{Y,\xi,T,2}+\phi_{Y,0}I_{Y,X,T}-\frac{\phi^2_{Y,0}}{2} K_{Y,\xi,T,1}\right)\notag\\
&\quad\quad\times \exp\left(\phi^2_{X,0}K_{X,\xi,T,2}+\phi_{X,0}I_{X,T}-\frac{\phi^2_{X,0}}{2}
K_{X,\xi,T,1}\right).\notag
\end{align}

The expressions (\ref{eq:B_L_T}) and (\ref{eq:B_U_T}) have the same asymptotic form. We first provide
the intuitive idea and then rigorously prove our result on asymptotic approximation. 
Note that, by (H3) and (H7), (\ref{eq:v_Y_lower_bound}), (\ref{eq:v_Y_upper_bound}),
(\ref{eq:v_X_lower_bound}), (\ref{eq:v_X_upper_bound}), 
the facts that $\frac{1}{b_T-a_T}\int_{a_T}^{b_T}\lambda^2(s)ds\rightarrow 0$ as $T\rightarrow\infty$, 
and $\xi$ is a finite random variable, that
$K_{Y,\xi,T,1}\stackrel{a.s.}{\sim} (b_T-a_T)K_Y$, $K_{Y,\xi,T,2}\stackrel{a.s.}{\sim} (b_T-a_T)K_Y$, 
$K_{X,\xi,T,1}\stackrel{a.s.}{\sim} (b_T-a_T)K_X$
and $K_{X,\xi,T,2}\stackrel{a.s.}{\sim} (b_T-a_T)K_X$, where, for any two random sequences 
$\left\{A_T:~T\geq 0\right\}$ and $\left\{B_T:~T\geq 0\right\}$,
$A_T\stackrel{a.s.}{\sim} B_T$ stands for $A_T/B_T\rightarrow 1$, almost surely, as $T\rightarrow\infty$.
Also, as we show, the distributions of $(b_T-a_T)^{-\frac{1}{2}}I_{Y,X,T}$ and $(b_T-a_T)^{-\frac{1}{2}}I_{X,T}$ 
are asymptotically normal with
zero means and variances $K_Y$ and $K_X$, respectively.
Heuristically substituting these in (\ref{eq:B_L_T}) and (\ref{eq:B_U_T}) yields the form
\begin{equation*}
\hat p_T(\theta_0)
=\exp\left(\frac{(b_T-a_T)K_Y\phi^2_{Y,0}}{2}+\phi_{Y,0}\sqrt{K_Y}\left(W_Y(b_T)-W_Y(a_T)\right)
+(b_T-a_T)K_X\phi^2_{X,0}\right).
\end{equation*}

\subsection{Modeled likelihood and its asymptotic approximation}
\label{subsec:modeled_likelihood}
Our modeled likelihood associated with the state space model described by (\ref{eq:data_sde2})
and (\ref{eq:state_space_sde2}) is given by:
\begin{align}
L_T(\theta)&=\int \exp\left(\phi_{Y} u_{Y|X,T}-\frac{\phi^2_{Y}}{2}v_{Y|X,T}\right)
\times \exp\left(\phi_{X} u_{X,T}-\frac{\phi^2_{X}}{2}v_{X,T}\right)dQ_{T,X}.\label{eq:L_T_3}
\end{align}

Using the same method of obtaining bounds of $p_T(\theta_0)$, we obtain the following
bounds for $L_T(\theta)$:
\begin{equation*}
\tilde B_{L,T}(\theta)\leq L_T(\theta)\leq \tilde B_{U,T}(\theta),
\end{equation*}
where
\begin{align}
\tilde B_{L,T}(\theta)&=E_{T,X}\left(\tilde Z_{L,T,\theta}(X)\right);\notag\\
\tilde B_{U,T}(\theta)&=E_{T,X}\left(\tilde Z_{U,T,\theta}(X)\right),\notag 
\end{align}
where
\begin{align}
\tilde Z_{L,T,\theta}(X)&=
\exp\left(\phi_Y\phi_{Y,0}K_{Y,\xi,T,1}+\phi_YI_{Y,X,T}-\frac{\phi^2_Y}{2} K_{Y,\xi,T,2}\right)\notag\\
&\quad\quad\times \exp\left(\phi_X\phi_{X,0}K_{X,\xi,T,1}+\phi_XI_{X,T}-\frac{\phi^2_X}{2}
K_{X,\xi,T,2}\right)\notag
\end{align}
and
\begin{align}
\tilde Z_{U,T,\theta}(X)&=\exp\left(\phi_Y\phi_{Y,0}K_{Y,\xi,T,2}+\phi_YI_{Y,X,T}
-\frac{\phi^2_Y}{2} K_{Y,\xi,T,1}\right)\notag\\
&\quad\quad\times \exp\left(\phi_X\phi_{X,0}K_{X,\xi,T,2}+\phi_XI_{X,T}-\frac{\phi^2_X}{2}
K_{X,\xi,T,1}\right).\notag
\end{align}
It follows as before that the modeled likelihood can be approximated as
\begin{align}
\hat L_T(\theta)
&=\exp\left((b_T-a_T)K_Y\phi_Y\phi_{Y,0}+\phi_Y\sqrt{K_Y}\left(W_Y(b_T)-W_Y(a_T)\right)\right.\notag\\
&\qquad\qquad\left.-\frac{(b_T-a_T)K_Y\phi^2_Y}{2}+(b_T-a_T)K_X\phi_X\phi_{X,0}\right).\notag
\end{align}

\subsection{A briefing on the formal results on the asymptotic approximations}

Formal proof of the results $p_T(\theta_0)\stackrel{a.s.}{\sim}\hat p_T(\theta_0)$ and
$L_T(\theta)\stackrel{a.s.}{\sim}\hat L_T(\theta)$ requires the following two additional assumptions:

\begin{itemize}
\item[(H9)] There exists an integer $k_0\geq 1$ such that 
$\sum_{T=1}^{\infty}\delta^{-2k_0}_T \left(b_T-a_T\right)^{k_0-1}\int_{a_T}^{b_T}\lambda^2(s)ds<\infty$,
where $\delta_T\downarrow 0$ as $T\rightarrow\infty$ is a specific sequence decreasing fast enough 
so that it satisfies, because of continuity of the exponential function, the following: for any $\epsilon>0$, 
\begin{align}
&\sum_{T=1}^{\infty}P\left(\left|I_{Y,X,T}-\sqrt{K_Y}\left(W_Y(b_T)-W_Y(a_T)\right)\right|\leq \delta_T,\right.\notag\\
&\qquad\qquad~\left.\left|\exp\left(I_{Y,X,T}\right)
-\exp\left(\sqrt{K_Y}\left(W_Y(b_T)-W_Y(a_T)\right)\right)\right|>\epsilon\right)
<\infty.
\label{eq:exp_as3}
\end{align}
Also assume that $E|\xi|^{2k_0}<\infty$.
\item[(H10)] $\underset{T>0}{\sup}~E\left(\frac{Z_{L,T,\theta_0}(X)}{\hat p_T(\theta_0)}\right)<\infty$, 
$\underset{T>0}{\sup}~E\left(\frac{Z_{U,T,\theta_0}(X)}{\hat p_T(\theta_0)}\right)<\infty$,
$\underset{T>0,~\theta\in\Theta}{\sup}~E\left(\frac{\tilde Z_{L,T,\theta}(X)}{\hat L_T(\theta)}\right)<\infty$
and  $\underset{T>0,~\theta\in\Theta}{\sup}~E\left(\frac{\tilde Z_{U,T,\theta}(X)}{\hat L_T(\theta)}\right)<\infty$.
\end{itemize}

The following lemma shows that under assumptions (H1) -- (H9),
$\exp\left(I_{Y,X,T}\right)$ and $\exp\left(I_{X,T}\right)$ are asymptotically
independent of $X$.
\begin{lemma}
\label{lemma:asymp_equal}
Under assumptions (H1) -- (H9),
\begin{align}
\left|\exp\left(I_{Y,X,T}\right)-\exp\left(\sqrt{K_Y}\left(W_Y(b_T)-W_Y(a_T)\right)\right)\right|
&\stackrel{a.s.}{\longrightarrow} 0; 
\label{eq:exp_I_Y_as}\\
\left|\exp\left(I_{X,T}\right)-\exp\left(\sqrt{K_X}\left(W_X(b_T)-W_X(a_T)\right)\right)\right|
&\stackrel{a.s.}{\longrightarrow} 0. 
\label{eq:exp_I_X_as}
\end{align}
\end{lemma}
The following corollary of Lemma \ref{lemma:asymp_equal} shows asymptotic normality of
the relevant quantities involved in the asymptotic approximations.
\begin{corollary}
\label{corollary:asymp_normal}
Since $(b_T-a_T)^{-\frac{1}{2}}\sqrt{K_Y}\left(W_Y(b_T)-W_Y(a_T)\right)$ and 
$(b_T-a_T)^{-\frac{1}{2}}\sqrt{K_X}\left(W_X(b_T)-W_X(a_T)\right)$ are normally distributed
with mean zero and variances $K_Y$ and $K_X$, respectively, it follows that
\begin{align}
(b_T-a_T)^{-\frac{1}{2}}I_{Y,X,T}\stackrel{a.s.}{\longrightarrow}N(0,K_Y);\notag\\
(b_T-a_T)^{-\frac{1}{2}}I_{X,T}\stackrel{a.s.}{\longrightarrow}N(0,K_X).\notag 
\end{align}
\end{corollary}

Finally, our asymptotic approximation result is given by the following theorem, which requires
assumptions (H1) -- (H10).
\begin{theorem}
\label{theorem:asymp_approx}
Assume (H1) -- (H10). Then 
\begin{align}
p_T(\theta_0)&\stackrel{a.s.}{\sim}\hat p_T(\theta_0);\label{eq:p_T_asymp}\\
L_T(\theta)&\stackrel{a.s.}{\sim}\hat L_T(\theta),~\mbox{for all}~\theta\in\Theta.\label{eq:L_T_asymp}
\end{align}
\end{theorem}

The proofs of Lemma \ref{lemma:asymp_equal} and Theorem \ref{theorem:asymp_approx}
are presented in the supplement.

\section{Convergence of the posterior distribution of $\theta$}
\label{sec:posterior_convergence}

In order to prove convergence of our posterior distribution we verify the conditions of the theorem
proved in \ctn{Shalizi09} which take account of dependence setups and misspecifications. 
The detailed assumptions in our state space $SDE$ context and Shalizi's theorem is provided in Section \ref{sec:assumptions_shalizi} of the supplement. 

\subsection{Further assumptions}
\label{subsec:further_assumptions}

Before proceeding further, we make the following assumptions
regarding $\psi_Y$ and $\psi_X$:
\begin{itemize}
\item[(H11)] 
\begin{enumerate}
\item[(i)] For every $\theta\in\Theta\cup\{\theta_0\}$, $\psi_Y(\theta)$ and $\psi_X(\theta)$ are finite
and satisfy $(\psi_Y(\theta_1),\psi_X(\theta_1))=(\psi_Y(\theta_2),\psi_X(\theta_2))$ implies
$\theta_1=\theta_2$. 
\item[(ii)]
$\left|\psi_Y\right|$ is coercive, that is, 
for every sequence $\{\theta_T:~T>0\}$ such that $\|\theta_T\|\rightarrow\infty$,
$\left|\psi_Y(\theta_T)\right|\rightarrow\infty$. 
\item[(iii)] For every sequence $\{\theta_T:~T>0\}$ such that $\|\theta_T\|\rightarrow\infty$,
$\left|\psi_Y(\theta_T)\right|^2(b_T-a_T)|K_{Y,j,T}-K_Y|\rightarrow 0$ and
$\left|\psi_X(\theta_T)\right|^2(b_T-a_T)|K_{X,j,T}-K_X|\rightarrow 0$, for $j=1,2$, and
$C_1(b_T-a_T)\leq\left(\psi_Y(\theta_T)-\psi_Y(\theta_0)\right)^8\leq C_2(b_T-a_T)$, 
for some constants $C_1,C_2>0$, as $T\rightarrow\infty$.
\item[(iv)] $\left|\psi_Y(\theta)\right|$ is assumed to have finite expectation with respect to the prior $\pi(\theta)$.
\item[(v)] $|\psi_X(\theta)|\leq |\psi_X(\theta_0)|$, for all $\theta\in\Theta$.
\item[(vi)] The first and second derivatives of $\psi_X$ vanish at $\theta=\theta_0$.
\item[(vii)] $\psi_Y$ and $\psi_X$ are at least thrice continuously differentiable.
\end{enumerate}
\end{itemize}

\subsection{Verification of the assumptions of Shalizi}
\label{subsec:verification_shalizi}

\subsubsection{Verification of (A1)}
\label{subsubsec:verify_A1}

Recall that our likelihood $L_T(\theta)$ is given by (\ref{eq:L_T_3}).
In the same way as the proof of the second part of Proposition 2 of \ctn{Maud12}, it can be proved that 
the first factor of the integrand
of (\ref{eq:L_T_3}) is a measurable function of $(\left\{Y(s);s\in[a_T,b_T]\right\},
\left\{X(s);s\in[a_T,b_T]\right\},\theta)$. 
Also, by the same result of \ctn{Maud12} the second factor of the integrand is a measurable function 
of $\left(\left\{X(s);s\in[a_T,b_T]\right\},\theta\right)$.
Thus, the integrand is a measurable function of\\ $(\left\{Y(s);s\in[a_T,b_T]\right\}, 
\left\{X(s);s\in[a_T,b_T]\right\},\theta)$. 
Since the associated measure spaces are $\sigma$-finite, $L_T(\theta)$ is 
clearly $\mathcal F_T\times\mathcal T$-measurable for
all $T>0$.

\subsubsection{Verification of (A2)}
\label{subsubsec:verify_A2}
We consider the likelihood ratio $R_T(\theta)$ given by (\ref{eq:R_T}).
Using Theorem \ref{theorem:asymp_approx} we obtain that
\begin{align}
\frac{1}{b_T-a_T}\log R_T(\theta)&\stackrel{a.s.}{\sim}
-\frac{K_Y}{2}\left(\phi_Y-\phi_{Y,0}\right)^2
+\sqrt{K_Y}\left(\phi_Y-\phi_{Y,0}\right)\frac{\left(W_Y(b_T)-W_Y(a_T)\right)}{b_T-a_T}\notag\\
&\qquad\qquad-\frac{K_X}{2}\left(\phi_X-\phi_{X,0}\right)^2
+\frac{K_X}{2}\left(\phi^2_X-\phi^2_{X,0}\right).
\label{eq:log_R_T_approx}
\end{align}

Since $\frac{W_Y(b_T)-W_Y(a_T)}{b_T-a_T}\stackrel{a.s.}{\longrightarrow}0$,
it follows that, almost surely,
\begin{equation}
\frac{1}{b_T-a_T}\log R_T(\theta)\rightarrow -\frac{1}{2}\left[K_Y\left(\phi_Y-\phi_{Y,0}\right)^2+
K_X\left(\phi_X-\phi_{X,0}\right)^2+K_X\left(\phi^2_{X,0}-\phi^2_{X}\right)\right].
\label{eq:log_R_T_limit}
\end{equation}
Let
\begin{align}
h(\theta)&=\frac{1}{2}\left[K_Y\left(\phi_Y-\phi_{Y,0}\right)^2+K_X\left(\phi_X-\phi_{X,0}\right)^2
+K_X\left(\phi^2_{X,0}-\phi^2_{X}\right)\right]\notag\\
&=\frac{1}{2}\left[K_Y\left(\psi_Y(\theta)-\psi_Y(\theta_0)\right)^2
+K_X\left(\psi_X(\theta)-\psi_X(\theta_0)\right)^2
+K_X\left(\psi^2_X(\theta_0)-\psi^2_{X}(\theta)\right)\right].
\label{eq:h}
\end{align}
Note that due to (H11) (v), $h(\theta)\geq 0$, for all $\theta\in\Theta$.
Thus (A2) holds. 


\subsubsection{Verification of (A3)}
\label{subsubsec:R_T_asymp}

We now obtain the limit of the quantity 
$$\frac{1}{b_T-a_T}E_{\theta_0}\left(\log\frac{p_T(\theta_0)}{L_T(\theta)}\right)
=-\frac{1}{b_T-a_T}E_{\theta_0}\left(\log R_T(\theta)\right),$$
where $E_{\theta_0}$ is the expectation with respect to the true likelihood $p_T(\theta_0)$.
Proceeding in the same way as in the case of $R_T(\theta)$ and noting that 
$E_{\theta_0}\left(W_Y(b_T)-W_Y(a_T)\right)=0$,
it is easy to see that
\begin{equation*}
\frac{1}{b_T-a_T}E_{\theta_0}\left(\log\frac{p_T(\theta_0)}{L_T(\theta)}\right)\rightarrow h(\theta),
\end{equation*}
as $T\rightarrow\infty$.

\subsubsection{Verification of (A4)}
\label{subsubsec:A4}

To verify (A4) we reformulate the original parameter space $\Theta$
as $\Theta\setminus I$. Abusing notation, we continue to denote $\Theta\setminus I$ as $\Theta$.
Hence, the prior $\pi$ on $\Theta$ clearly satisfies $\pi(I)=0$.


\subsubsection{Verification of (A5) (i)}
\label{subsubsec:A5_1}
Now consider $\mathcal G_T=\left\{\theta\in\Theta:\left|\psi_Y(\theta)\right|\leq \exp(\beta (b_T-a_T))\right\}$, where
$\beta$ is chosen such that $\beta>2h\left(\Theta\right)$.
Coerciveness of $\|\psi_Y\|$ implies compactness of $\mathcal G_T$, for every $T>0$.

The above definition of $\mathcal G_T$  clearly implies $\mathcal G_T\rightarrow\Theta$. Also,
\begin{align}
\pi\left(\mathcal G_T\right)&>1-E\left(\left|\psi_Y(\theta)\right|\right)\exp\left(-\beta (b_T-a_T)\right)\notag\\
&=1-\alpha\exp\left(-\beta (b_T-a_T)\right),\notag
\end{align}
where the first inequality is due to Markov's inequality and $\alpha=E\left(\left|\psi_Y(\theta)\right|\right)>0$.
The expectation, which is with respect to the prior $\pi$, exists by (H11) (iv).

\subsubsection{Verification of (A5) (ii)}
\label{subsubsec:A5_2}
We now show that convergence of (\ref{eq:log_R_T_limit}) is uniform in $\theta$ over $\mathcal G_T\setminus I$.
First note that $\mathcal G_T\setminus I=\mathcal G_T$, since we have already removed $I$ from $\Theta$.
Now note that, because of compactness of $\mathcal G_T$
and continuity of $\left|\frac{1}{b_T-a_T}\log R_T(\theta)+h(\theta)\right|$ in $\theta$, 
there exists $\theta_T\in\mathcal G_T$ such that
\begin{align}
&\underset{\theta\in\mathcal G_T\setminus I}{\sup}~\left|\frac{1}{b_T-a_T}\log R_T(\theta)+h(\theta)\right|
=\left|\frac{1}{b_T-a_T}\log R_T(\theta_T)+h(\theta_T)\right|.\label{eq:uniform_convergence}
\end{align}
Note that $\theta_T$ depends upon the data. However, under the additional
condition (H11) (iii), it is clear from the proof of Theorem \ref{theorem:asymp_approx}
(see Section \ref{sec:proof_theorem_asymp_approx} of the supplement) that 
our asymptotic approximation
of $L_T(\theta_T)$ remains valid even in this case. Formally,
\begin{theorem}
\label{theorem:asymp_approx2}
Assume (H1) -- (H10) and (H11) (iii). 
Consider any, perhaps, data-dependent sequence $\left\{\theta_T:~T>0\right\}$, 
where either $\|\theta_T\|$ remains finite almost surely or $\|\theta_T\|\rightarrow\infty$, almost surely, 
as $T\rightarrow\infty$.
Then $L_T(\theta_T)\stackrel{a.s.}{\sim}\hat L_T(\theta_T)$.
\end{theorem}
The above theorem guarantees that (\ref{eq:uniform_convergence}) admits the following approximation: 
\begin{align}
\left|\frac{1}{b_T-a_T}\log R_T(\theta_T)+h(\theta_T)\right|
\stackrel{a.s.}{\sim}
\sqrt{K_Y}\left|\frac{\left(\psi_Y(\theta_T)-\psi_Y(\theta_0)\right)}{\sqrt{b_T-a_T}}
\times\frac{W_Y(b_T)-W_Y(a_T)}{\sqrt{b_T-a_T}}\right|.
\label{eq:uniform_convergence2}
\end{align}
By Corollary \ref{corollary:asymp_normal} and (H11) (iii), the right hand side of (\ref{eq:uniform_convergence2})
goes to zero almost surely, as $T\rightarrow\infty$. 
Hence the convergence of (\ref{eq:log_R_T_limit}) is uniform in $\theta$ over $\mathcal G_T\setminus I$.

\subsubsection{Verification of (A5) (iii)}
\label{subsubsec:A5_3}
We now show that $h\left(\mathcal G_T\right)\rightarrow h\left(\Theta\right)$, as $T\rightarrow\infty$.
Due to compactness of $\mathcal G_T$ and continuity of $h(\theta)$, it follows that there exists
$\tilde\theta_T\in\mathcal G_T$ such that $h\left(\mathcal G_T\right)=h(\tilde\theta_T)$.
Also, since $\mathcal G_T$ is a non-decreasing sequence of sets, $h(\tilde\theta_T)$ 
is non-increasing in $T$. Since $\mathcal G_T\rightarrow\Theta$, it follows that 
$h\left(\mathcal G_T\right)\rightarrow h\left(\Theta\right)$, as $T\rightarrow\infty$.

\subsubsection{Verification of (A6)}
\label{subsubsec:A6}

Under (A1) -- (A3), which we have already verified,
it holds that (see equation (18) of \ctn{Shalizi09}) for any fixed $\mathcal G$ of the sequence $\mathcal G_T$, 
for any $\epsilon>0$ and
for sufficiently large $T$,
\begin{equation*}
\frac{1}{b_T-a_T}\log\int_{\mathcal G}R_T(\theta)\pi(\theta)d\theta
\leq -h(\mathcal G)+\epsilon+\frac{1}{b_T-a_T}\log\pi(\mathcal G).
\end{equation*}
It follows that $\tau(\mathcal G_T,\delta)$ is almost surely finite for all $T$ and $\delta$.
We now argue that for sufficiently large $T$, $\tau(\mathcal G_T,\delta)>(b_T-a_T)$ only finitely often
with probability one. 
By equation (41) of \ctn{Shalizi09},
\begin{align}
\sum_{T=1}^{\infty}P\left(\tau(\mathcal G_T,\delta)>(b_T-a_T)\right)
\leq \sum_{T=1}^{\infty}\sum_{m=T+1}^{\infty}
P\left(\frac{1}{b_m-a_m}\log\int_{\mathcal G_T}R_m(\theta)\pi(\theta)d\theta>\delta-h(\mathcal G_T)\right).
\label{eq:A6_1}
\end{align}
Now, by compactness of $\mathcal G_T$, $h(\mathcal G_T)=h(\tilde\theta_T)$, for
$\tilde\theta_T\in\mathcal G_T$, and 
by the mean value theorem for integrals,
$$\frac{1}{b_m-a_m}\log\int_{\mathcal G_T}R_m(\theta)\pi(\theta)d\theta
=\frac{1}{b_m-a_m}\log R_m(\hat\theta_T)\pi(\mathcal G_T),$$
for $\hat\theta_T\in\mathcal G_T$ depending upon the data, so that
$$
\frac{1}{b_m-a_m}\log\int_{\mathcal G_T}R_m(\theta)\pi(\theta)d\theta>\delta-h(\mathcal G_T)
$$
implies, since $h(\hat\theta_T)\geq h(\tilde\theta_T)$, that
$$
\frac{1}{b_m-a_m}\log R_m(\hat\theta_T)+h(\hat\theta_T)>\delta-\frac{1}{b_m-a_m}\log\pi(\mathcal G_T)>\delta.
$$
Thus, it follows from (\ref{eq:A6_1}) and Chebychev's inequality, that
\begin{align}
&\sum_{T=1}^{\infty}P\left(\tau(\mathcal G_T,\delta)>(b_T-a_T)\right)\notag\\
&\qquad\qquad\leq \sum_{T=1}^{\infty}\sum_{m=T+1}^{\infty}
P\left(\left|\frac{1}{b_m-a_m}\log R_m(\hat\theta_T)+h(\hat\theta_T)\right|>\delta\right)\notag\\
&\qquad\qquad\leq \sum_{T=1}^{\infty}\sum_{m=T+1}^{\infty}
\delta^{-8}E\left(\frac{1}{b_m-a_m}\log R_m(\hat\theta_T)+h(\hat\theta_T)\right)^{8}.
\label{eq:A6_2}
\end{align}
From (\ref{eq:log_R_T_approx}) and (\ref{eq:h}) it is clear that
\begin{equation}
\frac{1}{b_m-a_m}\log R_m(\hat\theta_T)+h(\hat\theta_T)\stackrel{a.s.}{\sim}\sqrt{K_Y}\frac{\left(\psi_Y(\hat\theta_T)-\psi_Y(\theta_0)\right)}{\sqrt{b_m-a_m}}
\times\frac{W_Y(b_m)-W_Y(a_m)}{\sqrt{b_m-a_m}}
\label{eq:approx}
\end{equation}

Now, let $Z_m=\frac{1}{b_m-a_m}\log R_m(\hat\theta_T)+h(\hat\theta_T)$ and
$\tilde Z_m=\sqrt{K_Y}\frac{\left(\psi_Y(\hat\theta_T)-\psi_Y(\theta_0)\right)}{\sqrt{b_m-a_m}}
\times\frac{W_Y(b_m)-W_Y(a_m)}{\sqrt{b_m-a_m}}$.
Then
\begin{align}
\frac{Z^{8}_m-\tilde Z^{8}_m}{E\left(\tilde Z^{8}_m\right)}
&=\frac{Z^{8}_m-\tilde Z^{8}_m}{\tilde Z^{8}_m}
\times\frac{\tilde Z^{8}_m}{E\left(\tilde Z^{8}_m\right)}
\stackrel{a.s.}{\longrightarrow}0~~\mbox{as}~~m\rightarrow\infty.
\label{eq:A6_3}
\end{align}
because, due to \ref{eq:approx} the first factor on the right hand side of (\ref{eq:A6_3})
tends to zero almost surely, while by (H11) (iii) the second factor is bounded above 
by a constant times standard normal distribution
raised to the power $6$.
It can be easily verified using (H11) (iii) that $\underset{m\geq 1}{\sup}~
E\left[\frac{Z^{8}_m-\tilde Z^{8}_m}{E\left(\tilde Z^{8}_m\right)}\right]^2<\infty$,
so that $\frac{Z^{8}_m-\tilde Z^{8}_m}{E\left(\tilde Z^{8}_m\right)}$ is uniformly integrable.
Hence, it follows from (\ref{eq:A6_3}) that 
\begin{align}
&\frac{E\left(Z^{8}_m\right)-E\left(\tilde Z^{8}_m\right)}{E\left(\tilde Z^{8}_m\right)}
\rightarrow 0,~~\mbox{as}~~m\rightarrow\infty.\notag
\end{align}
In other words, as $m\rightarrow\infty$,
\begin{equation}
E\left(Z^{8}_m\right)\stackrel{a.s.}{\sim} E\left(\tilde Z^{8}_m\right).
\label{eq:A6_4}
\end{equation}

Now note that for studying convergence of the double sum (\ref{eq:A6_2}), it is enough to
investigate convergence of  
\begin{equation*}
S_{T_0} = \sum_{T=T_0}^{\infty}\sum_{m=T+1}^{\infty}
E\left(\frac{1}{b_m-a_m}\log R_m(\hat\theta_T)+h(\hat\theta_T)\right)^{8},
\end{equation*}
for some sufficiently large $T_0$. By virtue of (\ref{eq:A6_4}) it is then enough to study
convergence of
\begin{align}
S_{T_0} &= \sum_{T=T_0}^{\infty}\sum_{m=T+1}^{\infty}
E\left(\tilde Z^{8}_m\right)\notag\\
&=\tilde c\sum_{T=T_0}^{\infty}\sum_{m=T+1}^{\infty}
\frac{\left(\psi_Y(\tilde\theta_T)-\psi_Y(\theta_0)\right)^8}{(b_m-a_m)^4},\notag
\end{align}
where $\tilde c~(>0)$ is a constant.
By (H11) (iii), 
for sufficiently large $T$,
$\left(\psi_Y(\tilde\theta_T)-\psi_Y(\theta_0)\right)^8\leq C_2(b_T-a_T)$, for some
$C_2>0$. 
Hence, 
\begin{equation*}
S_{T_0}\leq C_Y\sum_{T=T_0}^{\infty}\sum_{m=T+1}^{\infty}
\frac{b_T-a_T}{(b_m-a_m)^4},
\end{equation*}
where $C_Y~(>0)$ is a constant.
Now note that, since $(b_T-a_T)$ is increasing in $T$, $(b_{T_0+j}-a_{T_0+j})<(b_{T_0+j+1}-a_{T_0+j+1})$ for $j\geq 0$,
so that
\begin{align}
&\sum_{T=T_0}^{\infty}\sum_{m=T+1}^{\infty}\frac{b_T-a_T}{(b_m-a_m)^4}
=\frac{(b_{T_0}-a_{T_0})}{(b_{T_0+1}-a_{T_0+1})^4}
+\frac{(b_{T_0}-a_{T_0})+(b_{T_0+1}-a_{T_0+1})}{(b_{T_0+2}-a_{T_0+2})^4}\notag\\
&\qquad\qquad+\frac{(b_{T_0}-a_{T_0})+(b_{T_0+1}-a_{T_0+1})+(b_{T_0+2}-a_{T_0+2})}{(b_{T_0+3}-a_{T_0+3})^4}+\cdots\notag\\
&\leq \sum_{k=1}^{\infty}\frac{k}{(b_{T_0+k}-a_{T_0+k})^3}\notag\\
&\leq \sum_{k=1}^{\infty}\frac{k}{(T_0+k)^3}\leq \sum_{k=1}^{\infty}\frac{(T_0+k)}{(T_0+k)^3}
=\sum_{k=1}^{\infty}\frac{1}{(T_0+k)^2}\leq \sum_{k=1}^{\infty}\frac{1}{k^2}\notag\\
&<\infty.\notag
\end{align}
That is, $S_{T_0}<\infty$ for sufficiently large $T_0$.
In other words, (A6) holds.

\subsubsection{Verification of (A7)}
\label{subsubsec:A7}
For any set $A\subseteq\Theta$ with $\pi(A)>0$, it follows that $\mathcal G_T\cap A\rightarrow\Theta\cap A=A$.
Since $h\left(\mathcal G_T\cap A\right)$ is non-increasing as $T$ increases, it follows that
$h\left(\mathcal G_T\cap A\right)\rightarrow h\left(A\right)$,
as $T\rightarrow\infty$.

To summarize, we have the following theorem on posterior convergence of $\theta$.
\begin{theorem}
\label{theorem:bayesian_convergence}
Assume that the data was generated by the true model given by (\ref{eq:data_sde1}) and (\ref{eq:state_space_sde1}), but
modeled by (\ref{eq:data_sde2}) and (\ref{eq:state_space_sde2}). 
Assume (H1)--(H10) and (H11) (i) -- (v). 
For the prior $\pi$ on $\theta$, consider any set $A\in\mathcal T$ with $\pi(A)>0$ and $h(A)>h(\Theta)$. 
Then, almost surely,
\begin{equation*}
\underset{T\rightarrow\infty}{\lim}~\pi(A|\mathcal F_T)=0.
\end{equation*}
Moreover, if $\beta>2h(A)$ or $A\subset\cap_{k=T}^{\infty}\mathcal G_k$ for some $T$, then almost surely,
\begin{equation*}
\underset{T\rightarrow\infty}{\lim}~\frac{1}{b_T-a_T}\log\pi(A|\mathcal F_T)=-J(A).
\end{equation*}
\end{theorem}

\section{Consistency and asymptotic normality of the maximum likelihood estimator}
\label{sec:classical_asymptotics}


Now we make the following further assumption:
\begin{itemize}
\item[(H12)] The parameter space $\Theta$ is compact.
\end{itemize}

Let 
\begin{align}
g_{Y,T}(\theta)&=-\frac{K_Y}{2}\left(\psi_Y(\theta)-\psi_Y(\theta_0)\right)^2
+\sqrt{K_Y}\left(\psi_Y(\theta)-\psi_Y(\theta_0)\right)\frac{W_Y(b_T)-W_Y(a_T)}{b_T-a_T};\label{eq:g_Y}\\
g_{X,T}(\theta)&=-\frac{K_X}{2}\left(\psi_X(\theta)-\psi_X(\theta_0)\right)^2
+\frac{K_X}{2}\left(\psi^2_X(\theta)-\psi^2_X(\theta_0)\right).\label{eq:g_X}
\end{align}
Then note that
\begin{equation}
\underset{\theta\in\Theta}{\sup}~\left|\frac{1}{b_T-a_T}\log R_T(\theta)-g_{Y,T}(\theta)-g_{X,T}(\theta)\right|
=\left|\frac{1}{b_T-a_T}\log R_T(\theta^*_T)-g_{Y,T}(\theta^*_T)-g_{X,T}(\theta^*_T)\right|,\label{eq:mle_unicon}
\end{equation}
for some $\theta^*_T\in\Theta$ where $\theta^*_T$ is dependent on data. Proceeding in the same way as in Section \ref{subsubsec:A5_2} it is easily seen that
(\ref{eq:mle_unicon}) tends to zero almost surely with respect to both $Y$ and $X$, as $T\rightarrow\infty$.
Hence, the maximum likelihood estimator ($MLE$) can be approximated by maximizing the function
$$\tilde g_T(\theta)=g_{Y,T}(\theta)+g_{X,T}(\theta)$$ with respect to $\theta$.

\subsection{Strong consistency of the maximum likelihood estimator of $\theta$}
\label{subsec:mle_strong_consistency}

Observe that for $k=1,\ldots,d$,
\begin{align}
\frac{\partial\tilde g_T(\theta)}{\partial\theta_k}&=-K_Y\left(\psi_Y(\theta)-\psi_Y(\theta_0)\right)
\frac{\partial\psi_Y(\theta)}{\partial\theta_k}
-K_X\left(\psi_X(\theta)-\psi_X(\theta_0)\right)\frac{\partial\psi_X(\theta)}{\partial\theta_k}\notag\\
&\qquad+K_X\psi_X(\theta)\frac{\partial\psi_X(\theta)}{\partial\theta_k}
+\sqrt{K_Y}\frac{\partial\psi_Y(\theta)}{\partial\theta_k}\frac{W_Y(b_T)-W_Y(a_T)}{b_T-a_T}.\notag
\end{align}
Let $$\tilde g'_T(\theta)=\left(\frac{\partial\tilde g_T(\theta)}{\partial\theta_1},\ldots,
\frac{\partial\tilde g_T(\theta)}{\partial\theta_d}\right)^T.$$
Also, let
$\tilde g''_T(\theta)
=\left(\begin{array}{cccc}
\frac{\partial^2\tilde g_T(\theta)}{\partial\theta^2_1}&
\frac{\partial^2\tilde g_T(\theta)}{\partial\theta_1\partial\theta_2}&\cdots&
\frac{\partial^2\tilde g_T(\theta)}{\partial\theta_1\partial\theta_d}\\
\frac{\partial^2\tilde g_T(\theta)}{\partial\theta_2\partial\theta_1}&
\frac{\partial^2\tilde g_T(\theta)}{\partial\theta^2_2}&\cdots&
\frac{\partial^2\tilde g_T(\theta)}{\partial\theta_2\partial\theta_d}\\
\cdots & \cdots & \cdots & \cdots\\
\frac{\partial^2\tilde g_T(\theta)}{\partial\theta_d\partial\theta_1}&
\frac{\partial^2\tilde g_T(\theta)}{\partial\theta_d\partial\theta_2}&\cdots&
\frac{\partial^2\tilde g_T(\theta)}{\partial\theta^2_d}
\end{array}\right)$ denote the matrix with $(j,k)$-th element given by
\begin{align}
\frac{\partial^2 \tilde g_T(\theta)}{\partial\theta_j\partial\theta_k}&
= -K_Y\left[\frac{\partial\psi_Y(\theta)}{\partial\theta_j}\frac{\partial\psi_Y(\theta)}{\partial\theta_k}
+\left(\psi_Y(\theta)-\psi_Y(\theta_0)\right)\frac{\partial^2\psi_Y(\theta)}{\partial\theta_j\partial\theta_k}\right]
\notag\\
&\quad
-K_X\left[\frac{\partial\psi_X(\theta)}{\partial\theta_j}\frac{\partial\psi_X(\theta)}{\partial\theta_k}
+\left(\psi_X(\theta)-\psi_X(\theta_0)\right)\frac{\partial^2\psi_X(\theta)}{\partial\theta_j\partial\theta_k}\right]\notag\\
&\quad+K_X\left[\frac{\partial\psi_X(\theta)}{\partial\theta_j}\frac{\partial\psi_X(\theta)}{\partial\theta_k}
+\psi_X(\theta)\frac{\partial^2\psi_X(\theta)}{\partial\theta_j\partial\theta_k}\right]\notag\\
&\quad+\sqrt{K_Y}\frac{\partial^2\psi_Y(\theta)}{\partial\theta_j\partial\theta_k}
\frac{W_Y(b_T)-W_Y(a_T)}{b_T-a_T}.\notag
\end{align}
Note that by (H11) (vi),
\begin{align}
\left[\frac{\partial\tilde g_T(\theta)}{\partial\theta_k}\right]_{\theta=\theta_0}
&=\sqrt{K_Y}\left[\frac{\partial\psi_Y(\theta)}{\partial\theta_k}\right]_{\theta=\theta_0}\frac{W_Y(b_T)-W_Y(a_T)}{b_T-a_T}
\label{eq:g1}\\
\left[\frac{\partial^2 \tilde g_T(\theta)}{\partial\theta_j\partial\theta_k}\right]_{\theta=\theta_0}
&=-K_Y\left[\frac{\partial\psi_Y(\theta)}{\partial\theta_j}
\frac{\partial\psi_Y(\theta)}{\partial\theta_k}\right]_{\theta=\theta_0}
+\sqrt{K_Y}\left[\frac{\partial^2\psi_Y(\theta)}{\partial\theta_j\partial\theta_k}\right]_{\theta=\theta_0}
\frac{W_Y(b_T)-W_Y(a_T)}{b_T-a_T}.
\label{eq:g2}
\end{align}

Letting $\hat\theta_{T}$ denote the $MLE$, note that
\begin{equation}
0=\tilde g'_T(\hat\theta_{T})=\tilde g'_T(\theta_0)+\tilde g''_T(\theta^*_T)(\hat\theta_T-\theta_0),
\label{eq:taylor1}
\end{equation}
where $\theta^*_T$ lies between $\theta_0$ and $\hat\theta_T$. 
From (\ref{eq:g2}) it is clear that 
$$\left[\frac{\partial^2\tilde g_T(\theta)}{\partial\theta_j\partial\theta_k}\right]_{\theta=\theta_0} 
\stackrel{a.s.}{\longrightarrow}
-K_Y\left[\frac{\partial\psi_Y(\theta)}{\partial\theta_j}\frac{\partial\psi_Y(\theta)}
{\partial\theta_k}\right]_{\theta=\theta_0},
$$
as $T\rightarrow\infty$.
Let $\mathcal I(\theta)$ denote the matrix with $(j,k)$-th element given by
\begin{equation*}
\left\{\mathcal I(\theta)\right\}_{jk}=K_Y\left[\frac{\partial\psi_Y(\theta)}{\partial\theta_j}\frac{\partial\psi_Y(\theta)}
{\partial\theta_k}\right].
\end{equation*}
From (\ref{eq:g1}) it is obvious that $\left\{\mathcal I(\theta_0)\right\}_{jk}$ is the covariance between the $j$-th
and the $k$-th components of $\sqrt{b_T-a_T}\tilde g'_T(\theta_0)$, and so $\mathcal I(\theta_0)$ is non-negative definite.
We make the following assumptions:
\begin{itemize}
\item[(H13)] The true value $\theta_0\in \mbox{int} (\Theta)$, where by $\mbox{int} (\Theta)$ we mean
the interior of $\Theta$.
\item[(H14)] The matrix $\mathcal I(\theta)$ is positive definite for $\theta\in \mbox{int} (\Theta)$.
\end{itemize}
Hence, from (\ref{eq:taylor1}) we obtain, after pre-multiplying both sides of the relevant equation with
$\mathcal I^{-1}(\theta^*_T)$, the following:
\begin{equation}
-\mathcal I^{-1}(\theta^*_T)\tilde g''_T(\theta^*_T)\left(\hat\theta_T-\theta_0\right)
=\mathcal I^{-1}(\theta^*_T)\tilde g'_T(\theta_0).
\label{eq:mle1}
\end{equation}

Since as $T\rightarrow\infty$, $\tilde g'_T(\theta_0)\stackrel{a.s.}{\longrightarrow}0$ 
and $-\mathcal I^{-1}(\theta^*_T)\tilde g''_T(\theta^*_T)\stackrel{a.s.}{\longrightarrow} \mathfrak I_d$,
$\mathfrak I_d$ being the identity matrix of order $d$, it hold that
\begin{equation}
\hat\theta_T\stackrel{a.s.}{\longrightarrow}\theta_0,
\label{eq:mle_theta}
\end{equation}
as $T\rightarrow\infty$,
showing that the $MLE$ is strongly consistent.
The result can be formalized as the following theorem.
\begin{theorem}
\label{theorem:mle_strong_consistency1}
Assume that the data was generated by the true model given by (\ref{eq:data_sde1}) and (\ref{eq:state_space_sde1}), but
modeled by (\ref{eq:data_sde2}) and (\ref{eq:state_space_sde2}).
Assume conditions (H1)--(H14). Then the $MLE$ of $\theta$ is strongly consistent in the sense that
(\ref{eq:mle_theta}) holds. 
\end{theorem}

\subsection{Asymptotic normality of the maximum likelihood estimator of $\theta$}
\label{subsec:mle_normality}

Since $\hat\theta_T\stackrel{a.s.}{\longrightarrow}\theta_0$ and $\theta^*_T$ lies between $\theta_0$
and $\hat\theta_T$, it follows that $\theta^*_T\stackrel{a.s.}{\longrightarrow}\theta_0$ as $T\rightarrow\infty$.
This, and the fact that $(W_Y(b_T)-W_Y(a_T))/\sqrt{b_T-a_T}\sim N(0,1)$, guarantee that
\begin{equation*}
-\sqrt{b_T-a_T}\mathcal I^{-1}(\theta^*_T)\tilde g'_T(\theta_0)
\stackrel{\mathcal L}{\longrightarrow} N(\bzero,\mathcal I^{-1}(\theta_0)),
\end{equation*}
where $``\stackrel{\mathcal L}{\longrightarrow}"$ denotes convergence in distribution.
From (\ref{eq:mle1}) it then follows, using the fact 
$\mathcal I^{-1}(\theta^*_T)\tilde g''_T(\theta^*_T)\stackrel{a.s.}{\longrightarrow} \mathfrak I_d$, that
\begin{equation}
\sqrt{b_T-a_T}\left(\hat\theta_T-\theta_0\right)
\stackrel{\mathcal L}{\longrightarrow}N_d\left(0,\mathcal I^{-1}(\theta_0)\right).
\label{eq:clt2}
\end{equation}
Thus, we can present the following theorem.
\begin{theorem}
\label{theorem:mle_normality1}
Assume that the data was generated by the true model given by (\ref{eq:data_sde1}) and (\ref{eq:state_space_sde1}), but
modeled by (\ref{eq:data_sde2}) and (\ref{eq:state_space_sde2}).
Assume conditions (H1)--(H14). Then the $MLE$ of $\theta$ is asymptotically normal in the sense that
(\ref{eq:clt2}) holds. 
\end{theorem}

%

\section{Asymptotic posterior normality}
\label{sec:posterior_normal}

Let $\ell_T(\theta)=\log L_T(\theta)$ stand for the log-likelihood, and let 
\begin{equation*}
\Sigma^{-1}_T=\left\{\begin{array}{cc}
-\ell''_T(\hat\theta_T) & \mbox{if the inverse and}\ \ \hat\theta_T\ \ \mbox{exist}\\
\mathfrak I_d & \mbox{if not},
\end{array}\right.
\end{equation*}
where for any $z$,
\begin{equation*}
\ell''_T(z)=\left(\left(\frac{\partial^2}{\partial\theta_i\partial\theta_j}\ell_T(\theta)\bigg\vert_{\theta=z}\right)\right).
\end{equation*}
Thus, $\Sigma^{-1}_T$ is the observed Fisher's information matrix.

\subsection{Regularity conditions and a theorem of \ctn{Schervish95}}
\label{subsec:regularity_conditions}
\begin{itemize}
\item[(1)] The parameter space is $\Theta\subseteq\mathbb R^d$ for some finite $d$.
\item[(2)] $\theta_0$ is a point interior to $\Theta$.
\item[(3)] The prior distribution of $\theta$ has a density with respect to Lebesgue measure
that is positive and continuous at $\theta_0$.
\item[(4)] There exists a neighborhood $\mathcal N_0\subseteq\Theta$ of $\theta_0$ on which
$\ell_T(\theta)= \log L_T(\theta)$ is twice continuously differentiable with 
respect to all co-ordinates of $\theta$, 
$a.s.$ $[P_{\theta_0}]$.
\item[(5)] The largest eigenvalue of $\Sigma_T$ goes to zero in probability.
\item[(6)] For $\delta>0$, define $\mathcal N_0(\delta)$ to be the open ball of radius $\delta$ around $\theta_0$.
Let $\rho_T$ be the smallest eigenvalue of $\Sigma_T$. If $\mathcal N_0(\delta)\subseteq\Theta$, there exists
$K(\delta)>0$ such that
\begin{equation}
\underset{T\rightarrow\infty}{\lim}~P_{\theta_0}\left(\underset{\theta\in\Theta\backslash\mathcal N_0(\delta)}{\sup}~
\rho_T\left[\ell_T(\theta)-\ell_T(\theta_0)\right]<-K(\delta)\right)=1.
\label{eq:extra1}
\end{equation}
\item[(7)] For each $\epsilon>0$, there exists $\delta(\epsilon)>0$ such that
\begin{equation}
\underset{T\rightarrow\infty}{\lim}~P_{\theta_0}\left(\underset{\theta\in\mathcal N_0(\delta(\epsilon)),\|\gamma\|=1}{\sup}~
\left\vert 1+\gamma^T\Sigma^{\frac{1}{2}}_T\ell''_T(\theta)\Sigma^{\frac{1}{2}}_T\gamma\right\vert<\epsilon\right)=1.
\label{eq:extra2}
\end{equation}
\end{itemize}

\begin{theorem}[\ctn{Schervish95}]
\label{theorem:theorem4}
Assume the above seven regularity conditions. 
Then denoting $\Psi_T=\Sigma^{-1/2}_T\left(\theta-\hat\theta_T\right)$, for each compact subset 
$B$ of $\mathbb R^d$ and each $\epsilon>0$, the following holds:
\begin{equation*}
\lim_{T\rightarrow\infty}P_{\theta_0}
\left(\sup_{\Psi_T\in B}\left\vert\pi(\Psi_T\vert\mathcal F_T)-\varrho(\Psi_T)\right\vert>\epsilon\right)=0,
\end{equation*}
where $\varrho(\cdot)$ denotes the density of the standard normal distribution.
\end{theorem}

\subsection{Verification of the seven regularity conditions for posterior normality}
\label{subsec:verification_posterior_normality}

Also we assume that $\Theta$ is compact (assumption (H11)) which enables us to uniformly approximate
$\frac{1}{b_T-a_T}\log R_T(\theta)$ by $g_{Y,T}(\theta)+g_{X,T}(\theta)$ for $\theta\in\Theta$; 
see Section \ref{sec:classical_asymptotics}.
As a consequence, $\frac{1}{b_T-a_T}\ell_T(\theta)$ can be uniformly approximated by 
$g_{Y,T}(\theta)+g_{X,T}(\theta)+\frac{1}{b_T-a_T}\log p_T(\theta_0)$, for $\theta\in\Theta$.
Let
\begin{equation*}
\frac{1}{b_T-a_T}\tilde\ell_T(\theta)=g_{Y,T}(\theta)+g_{X,T}(\theta)+\frac{1}{b_T-a_T}\log p_T(\theta_0).
\end{equation*}
Henceforth, we shall be working with $ \frac{1}{b_T-a_T}\tilde\ell_T(\theta)$ whenever convenient. With this, the first four
regularity conditions presented in Section \ref{subsec:regularity_conditions} trivially hold.

To verify regularity condition (5), 
note that, since $\hat\theta_T\stackrel{a.s.}{\longrightarrow}\theta_0$, 
\begin{equation*}
\frac{1}{b_T-a_T}\tilde\ell''_T(\hat\theta_T)\stackrel{a.s.}{\longrightarrow}
-\mathcal I(\theta_0).
\end{equation*}
Hence, almost surely, $$\Sigma^{-1}_T\sim (b_T-a_T)\times\mathcal I(\theta_0),$$
so that $$\Sigma_T\stackrel{a.s.}{\longrightarrow} 0,$$ as $T\rightarrow\infty$. Thus, regularity condition (5) holds.

For verifying condition (6), observe that
\begin{equation*}
\rho_T\left[\ell_T(\theta)-\ell_T(\theta_0)\right]=\rho_T(b_T-a_T)\times\frac{1}{b_T-a_T}\log R_T(\theta),
\end{equation*}
where $\rho_T(b_T-a_T)\rightarrow c$, for some $c>0$ and, due to (\ref{eq:log_R_T_limit}),
\begin{align}
&\rho_T\left[\ell_T(\theta)-\ell_T(\theta_0)\right]\notag\\
&\quad\stackrel{a.s.}{\longrightarrow}
-\frac{c}{2}\left[K_Y\left(\psi_Y(\theta)-\psi_Y(\theta_0)\right)^2+
K_X\left(\psi_X(\theta)-\psi_X(\theta_0)\right)^2
+K_X\left(\psi^2_X(\theta_0)-\psi^2_X(\theta)\right)\right],
\label{eq:6_2}
\end{align}
for all $\theta\in \Theta\setminus\mathcal N_0(\delta)$.
Now note that
\begin{align}
&\underset{T\rightarrow\infty}{\lim}~P_{\theta_0}\left(\underset{\theta\in\Theta\backslash\mathcal N_0(\delta)}{\sup}~
\rho_T\left[\ell_T(\theta)-\ell_T(\theta_0)\right]<-K(\delta)\right)\notag\\
&\geq \underset{T\rightarrow\infty}{\lim}~P_{\theta_0}\left(\left(\rho_T(b_T-a_T)\right)
\times\frac{1}{b_T-a_T}\log R_T(\theta)<-K(\delta)~
\forall\theta\in\Theta\setminus\mathcal N_0(\delta)\right)\notag\\
&=1,
\label{eq:6_3}
\end{align}
the last step following due to (\ref{eq:6_2}). Thus, regularity condition (6) is verified.

For verifying condition (7), we note that $\theta\in\mathcal N_0(\delta(\epsilon))$ can be represented as 
$\theta=\theta_0+\delta_2\frac{\theta_0}{\|\theta_0\|}$, where $0<\delta_2\leq\delta(\epsilon)$. 
Hence, Taylor's series expansion around $\theta_0$ yields 
\begin{equation}
\frac{\tilde{\ell}''_T(\theta)}{b_T-a_T}=\frac{\tilde{\ell}''_T(\theta_0)}{b_T-a_T}+
\delta_2\frac{\tilde{\ell}'''_T(\theta^*)\theta_0}{(b_T-a_T)\|\theta_0\|},
\label{eq:last2}
\end{equation}
where $\theta^*$ lies between $\theta_0$ and $\theta$.
As $T\rightarrow\infty$, $\frac{\tilde{\ell}''_T(\theta_0)}{b_T-a_T}$ tends to 
$-\mathcal I(\theta_0)$, almost surely.
Now notice that
$$\frac{\left\|\tilde{\ell}'''_T(\theta^*)\theta_0\right\|}{(b_T-a_T)\|\theta_0\|}
\leq \frac{\|\tilde{\ell}'''_T(\theta^*)\|}{b_T-a_T}.$$ 
Because of (H11) (vii) and compactness of $\Theta$
it follows that $\frac{\|\tilde{\ell}'''_T(\theta^*)\|}{b_T-a_T}\rightarrow 0$ as $T\rightarrow\infty$.
Hence, it follows that $\tilde{\ell}''_T(\theta)=O\left(-(b_T-a_T)\times {\mathcal I}(\theta_0)+(b_T-a_T)\delta_2\right)$,
almost surely.
Since $\Sigma^{\frac{1}{2}}_T$ is asymptotically almost surely equivalent to $(b_T-a_T)^{-\frac{1}{2}}
{\mathcal I}^{-\frac{1}{2}}(\theta_0)$, condition (7) holds.
We summarize our result in the form of the following theorem.
\begin{theorem}
\label{theorem:theorem5}
Assume that the data was generated by the true model given by (\ref{eq:data_sde1}) and (\ref{eq:state_space_sde1}), but
modeled by (\ref{eq:data_sde2}) and (\ref{eq:state_space_sde2}).
Assume (H1) -- (H14). 
Then denoting $\Psi_T=\Sigma^{-1/2}_T\left(\theta-\hat\theta_T\right)$, for each compact subset 
$B$ of $\mathbb R^d$ and each $\epsilon>0$, the following holds:
\begin{equation*}
\lim_{T\rightarrow\infty}P_{\theta_0}
\left(\sup_{\Psi_T\in B}\left\vert\pi(\Psi_T\vert\mathcal F_T)-\varrho(\Psi_T)\right\vert>\epsilon\right)=0,
\end{equation*}
where $\varrho(\cdot)$ denotes the density of the standard normal distribution.
\end{theorem}

\section{Random effects models based on state space $SDE$s and a brief overview of the asymptotic results}
\label{sec:random_effects_brief}

\subsection{True and postulated systems of state space $SDE$s with random effects}
\label{subsec:systems_state_space}
We now consider the following ``true" random effects models based on state space $SDE$s:
for $i=1,\ldots,n$, and for $t\in[0,b_{T}]$,
\begin{align}
d Y_i(t)&=\phi_{Y_i,0} b_Y(Y_i(t),X_i(t),t)dt+\sigma_Y(Y_i(t),X_i(t),t)dW_{Y,i}(t);
\label{eq:data_sde3}\\
d X_i(t)&=\phi_{X_i,0}b_X(X_i(t),t)dt+\sigma_X(X_i(t),t)dW_{X,i}(t).
\label{eq:state_space_sde3}
\end{align}
In the above, $\phi_{Y_i,0}=\psi_{Y_i}(\theta_0)$ and $\phi_{X_i,0}=\psi_{X_i}(\theta_0)$, where
$\psi_{Y_i}$ and $\psi_{X_i}$ are known functions; $\theta_0$ is the true set of parameters.

Our modeled state space $SDE$ is given, for $t\in [0,b_{T}]$ by:
\begin{align}
d Y_i(t)&=\phi_{Y_i} b_Y(Y_i(t),X_i(t),t)dt+\sigma_Y(Y_i(t),X_i(t),t)dW_{Y,i}(t);
\label{eq:data_sde4}\\
d X_i(t)&=\phi_{X_i}b_X(X_i(t),t)dt+\sigma_X(X_i(t),t)dW_{X,i}(t),
\label{eq:state_space_sde4}
\end{align}
where $\phi_{Y_i}=\psi_{Y_i}(\theta)$ and $\phi_{X_i}=\psi_{X_i}(\theta)$. 
As before, we wish to learn about the set of parameters $\theta$. 
Note that for simplicity of our asymptotic analysis we assumed the same time interval
$[0,b_T]$ for $i=1,\ldots,n$.
We assume that $\psi_{Y_i}(\theta)\rightarrow\bar\psi_Y(\theta)$ and $\psi_{X_i}(\theta)\rightarrow\bar\psi_X(\theta)$, as $i\rightarrow\infty$, for all $\theta\in\Theta$.
Also, let $K_{Y,i}$ and $K_{X,i}$ be the relevant constants associated with (\ref{eq:data_sde4}) and (\ref{eq:state_space_sde4}), analogous to $K_Y$ and $K_X$
associated with (\ref{eq:data_sde2}) and (\ref{eq:state_space_sde2}), respectively. We assume that $K_{Y,i}\rightarrow\bar K_Y$ and $K_{X,i}\rightarrow\bar K_X$,
as $i\rightarrow\infty$.
Let $p_{T,i}(\theta_0)$ and $L_{T,i}(\theta)$ be the true and modeled likelihoods associated with the $i$-th state space $SDE$. 

\subsection{A brief overview of the main asymptotic results}
\label{subsec:asymp_overview_ss}
\subsubsection{Posterior convergence of $\theta$}
\label{subsubsec:bayesian_consistency_ss}


Here the true likelihood on $[a_{T},b_{T}]$ is of the form  
$\bar p_{n,T}(\theta_0)=\prod_{i=1}^np_{T,i}(\theta_0)\stackrel{a.s.}{\sim}\prod_{i=1}^n\hat p_{T,i}(\theta_0)$, where
\begin{equation*}
\hat p_{T,i}(\theta_0)
=\exp\left(\frac{(b_{T}-a_{T})K_{Y_i}\phi^2_{Y_i,0}}{2}+\phi_{Y_i,0}\sqrt{K_{Y_i}}
\left(W_{Y_i}(b_{T})-W_{Y_i}(a_{T})\right)
+(b_{T}-a_{T})K_{X_i}\phi^2_{{X_i},0}\right).
\end{equation*}

The modeled likelihood on $[a_T,b_T]$ is $\bar L_{n,T}(\theta)=\prod_{i=1}^nL_{T,i}(\theta)\stackrel{a.s.}{\sim}\prod_{i=1}^n\hat L_{T,i}(\theta)$, where
\begin{align}
\hat L_{T,i}(\theta)
&=\exp\left((b_T-a_T)K_{Y_i}\phi_{Y_i}\phi_{Y_i,0}+\phi_{Y_i}\sqrt{K_{Y_i}}\left(W_{Y_i}(b_T)-W_{Y_i}(a_T)\right)\right.\notag\\
&\qquad\qquad\left.-\frac{(b_T-a_T)K_{Y_i}\phi^2_{Y_i}}{2}+(b_T-a_T)K_{X_i}\phi_{X_i}\phi_{X_i,0}\right).\notag
\end{align}

Let $\bar R_{n,T}(\theta)=\frac{\bar L_{n,T}(\theta)}{\bar p_{n,T}(\theta_0)}$.
Then the following asymptotic equipartition property holds for the systems of state space $SDE$s:
\begin{equation*}
\underset{n\rightarrow\infty}{\lim}\underset{T\rightarrow\infty}{\lim}
~\frac{1}{n(b_T-a_T)}\log \bar R_{n,T}(\theta) = -\bar h(\theta),
\end{equation*}
almost surely, where
\begin{equation*}
\bar h(\theta) = \frac{1}{2}\left[\bar K_{Y}\left(\bar\psi_{Y}(\theta)-\bar\psi_{Y}(\theta_0)\right)^2+
\bar K_{X}\left(\bar\psi_{X}(\theta)-\bar\psi_{X}(\theta_0)\right)^2
+\bar K_{X}\left(\bar\psi^2_{X}(\theta_0)-\bar\psi^2_{X}(\theta)\right)\right].
\end{equation*}

We define, in our current context, the following:
\begin{align}
\bar h\left(A\right)&=\underset{\theta\in A}{\mbox{ess~inf}}~\bar h(\theta);\notag\\
\bar J(\theta)&=\bar h(\theta)-\bar h(\Theta);\notag\\
\bar J(A)&=\underset{\theta\in A}{\mbox{ess~inf}}~\bar J(\theta).\notag
\end{align}

We summarize our Bayesian convergence result in the form of the following theorem.
\begin{theorem}
\label{theorem:bayesian_convergence_ss_brief}
Let the true, data-generating model be given by (\ref{eq:data_sde3}) and (\ref{eq:state_space_sde3}), but let
the data be modeled by (\ref{eq:data_sde4}) and (\ref{eq:state_space_sde4}).
Consider any set $A\in\mathcal T$ with $\pi(A)>0$ and $\bar h\left(A\right)>\bar h\left(\Theta\right)$. 
Then under appropriate assumptions,
\begin{equation*}
\underset{n\rightarrow\infty,~T\rightarrow\infty}{\lim}~\pi(A|\bar{\mathcal F}_{n,T})=0,
\end{equation*}
where $\bar{\mathcal F}_{n,T}=\sigma\left(\left\{Y_{i}(s);i=1,\ldots,n;~s\in [a_T,b_T]\right\}\right)$.
If the set $A$ satisfies a technical condition, then we further have
\begin{equation}
\underset{n\rightarrow\infty,~T\rightarrow\infty}{\lim}~\frac{1}{n(b_T-a_T)}\log\pi(A|\bar{\mathcal F}_{n,T})=-\bar J(A).
\end{equation}
\end{theorem}
\begin{proof}[Sketch of the proof]
The proof easily follows using the asymptotic equipartition property for the systems of state space $SDE$s and construction of appropriate sieves of the form
$\bar{\mathcal G}_{n,T}=\left\{\theta:\left|\bar\psi_Y(\theta)\right|\leq\exp\left(\bar\beta n\left(b_T-a_T\right)\right)\right\}$, which have the desired
properties. Here $\bar\beta>2\bar h\left(\Theta\right)$.
\end{proof}

\subsubsection{Strong consistency of the $MLE$ of $\theta$}
\label{subsubsec:mle_strong_consistency_ss_brief}

\begin{theorem}
\label{theorem:mle_strong_consistency_ss_brief}
Let the true, data-generating model be given by (\ref{eq:data_sde3}) and (\ref{eq:state_space_sde3}), but let
the data be modeled by (\ref{eq:data_sde4}) and (\ref{eq:state_space_sde4}).
Then, under suitable regularity conditions, the $MLE$ of $\theta$, denoted by $\hat\theta_{n,T}$, is strongly consistent in the sense that 
$\hat\theta_{n,T}\stackrel{a.s.}{\longrightarrow}\theta_0$.
\end{theorem}
\begin{proof}[Sketch of the proof]
In this case, the $MLE$ can be approximated by maximizing 
$$\bar g_{n,T}(\theta)=\bar g_{Y,T}(\theta)+\bar g_{X,T}(\theta)$$ with respect to $\theta$,
where 
\begin{align}
\bar g_{Y,T}(\theta)&=-\frac{\bar K_Y}{2}\left(\bar\psi_Y(\theta)-\bar\psi_Y(\theta_0)\right)^2
+\sqrt{\bar K_Y}\left(\bar\psi_Y(\theta)-\bar\psi_Y(\theta_0)\right)\frac{1}{n}
\sum_{i=1}^n\frac{W_{Y_i}(b_T)-W_{Y_i}(a_T)}{b_T-a_T};\notag\\
\bar g_{X,T}(\theta)&=-\frac{\bar K_X}{2}\left(\bar\psi_X(\theta)-\bar\psi_X(\theta_0)\right)^2
+\frac{\bar K_X}{2}\left(\bar\psi^2_X(\theta)-\bar\psi^2_X(\theta_0)\right).\notag\\
\end{align}
The rest of the proof follows in the same lines as that of Theorem \ref{theorem:mle_strong_consistency1_brief}.
\end{proof}

\subsubsection{Asymptotic normality of the $MLE$ of $\theta$}
\label{subsubsec:mle_normality_ss_brief}
\begin{theorem}
\label{theorem:mle_normality_ss_brief}
Let the true, data-generating model be given by (\ref{eq:data_sde3}) and (\ref{eq:state_space_sde3}), but let
the data be modeled by (\ref{eq:data_sde4}) and (\ref{eq:state_space_sde4}).
Then, under suitable regularity conditions,
\begin{equation*}
\sqrt{n(b_T-a_T)}\left(\hat\theta_{n,T}-\theta_0\right)\stackrel{\mathcal L}{\longrightarrow}
N_d\left(0,{\mathcal I}^{-1}(\theta_0)\right),
\end{equation*}
as $n\rightarrow\infty$, $T\rightarrow\infty$.
In this case, the $(j,k)$-th element of the matrix
${\mathcal I}(\theta_0)$ is given by
\begin{equation*}
\left\{{\mathcal I}(\theta_0)\right\}_{jk}
=\bar K_Y\left[\frac{\partial\bar\psi_Y(\theta)}{\partial\theta_j}\frac{\partial\bar\psi_Y(\theta)}
{\partial\theta_k}\right]_{\theta=\theta_0}.
\end{equation*}
\end{theorem}
\begin{proof}[Sketch of the proof]
The proof of this result follows in the same way as that of Theorem \ref{theorem:mle_normality1_brief}.
\end{proof}

\subsubsection{Asymptotic posterior normality}
\label{subsubsec:verification_posterior_normality_ss}

We summarize our result on asymptotic posterior normality for systems of state space $SDE$s in the form of the following theorem.
\begin{theorem}
\label{theorem:theorem5_ss_brief}
Let the true, data-generating model be given by (\ref{eq:data_sde3}) and (\ref{eq:state_space_sde3}), but let
the data be modeled by (\ref{eq:data_sde4}) and (\ref{eq:state_space_sde4}).
Then denoting $\bar\Psi_{n,T}=\bar\Sigma^{-1/2}_{n,T}\left(\theta-\hat\theta_{n,T}\right)$, for each compact subset 
$B$ of $\mathbb R^d$ and each $\epsilon>0$, the following holds under appropriate regularity conditions:
\begin{equation*}
\lim_{n\rightarrow\infty,T\rightarrow\infty}P_{\theta_0}
\left(\sup_{\bar\Psi_{n,T}\in B}\left\vert\pi(\bar\Psi_{n,T}\vert\bar{\mathcal F}_{n,T})
-\varrho(\bar\Psi_T)\right\vert>\epsilon\right)=0.
\end{equation*}
\end{theorem}
\begin{proof}[Sketch of the proof]
In this case, $\ell_{n,T}(\theta)=\log L_{n,T}(\theta)$, 
can be uniformly approximated by 
\begin{equation*}
\frac{1}{n(b_T-a_T)}\bar\ell_{n,T}(\theta)=\bar g_{Y,T}(\theta)+\bar g_{X,T}(\theta)+\frac{1}{n(b_T-a_T)}\log \bar{p}_{nT}(\theta_0), 
\end{equation*}
for $\theta\in\Theta$.
The rest of the proof follows in the same way as that of Theorem \ref{theorem:theorem5_brief}.
\end{proof}

\section{Summary and discussion}
\label{sec:conclusion}

In this paper, we have investigated the asymptotic properties of the $MLE$ and the 
posterior distribution of the set of parameters associated with state space $SDE$s
and random effects state space $SDE$s. In particular, we have established posterior consistency
based on \ctn{Shalizi09} and asymptotic posterior normality based on \ctn{Schervish95}.
In addition, we have also established strong consistency and asymptotic normality of the $MLE$ 
associated with our state space $SDE$ models. Acknowledging the importance of discretization in
practical scenarios, we have shown (in Section \ref{sec:discrete_data} of the supplement) that our results go through even with 
discretized data.

In the case of our random effects $SDE$ models we only required independence of the state space models
for different individuals. That is, our approach and the results remain intact if the initial values
for the processes associated with the individuals are different. This is in contrast with the asymptotic 
works of \ctn{Maitra14a} and \ctn{Maitra14b} in the context of independent but non-identical 
random effects models for the individuals. Although not based on state space $SDE$s, their approach
required the simplifying assumption that the sequence of initial values is a convergent subsequence
of some sequence in some compact space. 

In fact, the relative simplicity of our current approach is due to the assumption of stochastic stability
of the latent processes of our models, the key concept that we adopted in our approach to alleviate 
the difficulties of the asymptotic problem at hand. Specifically, we adopted the conditions 
of Theorem 6.2 provided in \ctn{Mao11},
as sufficient conditions of our results. 
Indeed, there is a large literature on stochastic stability
of solutions of $SDE$s, with very many existing examples (see, for example, \ctn{Mao11} and the
references therein), which indicate that the assumption of stochastic
stability is not unrealistic. 

In our work we have assumed stochastic stability of $X$ only. If, in addition,
asymptotic stability of $Y$ is also 
assumed, then our results hold good by replacing (H3) in Section \ref{sec:assumptions}
with the following assumptions:
\begin{itemize}
\item[(H3(i))] $b_Y(0,0,t)=0=\sigma_Y(0,0,t)$ for all $t\geq 0$.
\item[(H3(ii))] For every $T>0$, there exist positive constants 
$K_{1,T}$, $K_{2,T}$, $\alpha_{1}$, $\alpha_{2}$, $\beta_{1}$ and $\beta_{2}$
such that for all $(x,t)\in\mathbb R\times [0,T]$,
$$K_{Y,1,T}\left(1-\alpha_{1}x^2-\beta_{1}y^2\right)\leq\frac{b^2_Y(y,x,t)}{\sigma^2_Y(y,x,t)}
\leq K_{Y,2,T}\left(1+\alpha_{2}x^2+\beta_{2}y^2\right),$$
where 
$K_{Y,1,T}\rightarrow K_Y$ and $K_{Y,2,T}\rightarrow K_Y$ and as $T\rightarrow\infty$; $K_Y$ being a 
positive constant as mentioned in (H3).
\end{itemize}
In this case the bounds of $\frac{b^2_Y(y,x,t)}{\sigma^2_Y(y,x,t)}$ are somewhat more general
than in (H3) in that they depend upon both $x$ and $y$, while in (H3) the bounds are independent of $y$.

To our knowledge, our work is the first time effort towards establishing asymptotic results 
in the context of state space $SDE$s, and the results we obtained are based on relatively
general assumptions which are satisfied by a large class of models. Since the notion of stochastic stability
is valid for any dimension of the associated $SDE$, it follows that our results admit straightforward
extension to high-dimensional state space $SDE$s. Corresponding results in the multidimensional 
extension of the random effects is provided briefly in Section \ref{sec:multidimensional} of the supplement.

As we mentioned in the introduction, our random effects state space $SDE$ model can not be interpreted 
as a {\it bona fide} random effects model from the classical perspective, and that introduction
of actual random effects would complicate our method of asymptotic investigation. 
Also, in this article we have assumed that the diffusion coefficients are free of parameters, which is
not a very realistic assumption. We are working on these issues currently, and will communicate our findings subsequently.

\section*{Acknowledgments}
We are sincerely grateful to the reviewer whose constructive suggestions have led to improved quality and presentation of the manuscript.
The first author gratefully acknowledges her NBHM Fellowship, Govt. of India.

\newpage

\renewcommand\thefigure{S-\arabic{figure}}
\renewcommand\thetable{S-\arabic{table}}
\renewcommand\thesection{S-\arabic{section}}

\setcounter{section}{0}
\setcounter{theorem}{0}

\begin{center}
{\LARGE\bf Supplementary Material}
\end{center}

Throughout, we refer to our main manuscript 
as MB. 

\section{Assumptions and theorem of Shalizi in the context of our state space $SDE$}
\label{sec:assumptions_shalizi}

\begin{itemize}
\item[(A1)] Consider the following likelihood ratio:
\begin{equation}
R_T(\theta)=\frac{L_T(\theta)}{p_T(\theta_0)}.
\label{eq:R_T}
\end{equation}
Assume that $R_T(\theta)$ is $\mathcal F_T\times \mathcal T$-measurable for all $T>0$.
\end{itemize}

\begin{itemize}
\item[(A2)] For each $\theta\in\Theta$, the generalized or relative asymptotic equipartition property holds, and so,
almost surely,
\begin{equation*}
\underset{T\rightarrow\infty}{\lim}~\frac{1}{b_T-a_T}\log R_T(\theta)=-h(\theta),
\end{equation*}
where $h(\theta)$ is given in (A3) below.
\end{itemize}

\begin{itemize}
\item[(A3)] For every $\theta\in\Theta$, the Kullback-Leibler divergence rate
\begin{equation}
h(\theta)=\underset{T\rightarrow\infty}{\lim}~\frac{1}{b_T-a_T}E_{\theta_0}\left(\log\frac{p_T(\theta_0)}{L_T(\theta)}\right).
\label{eq:A2}
\end{equation}
exists (possibly being infinite) and is $\mathcal T$-measurable. In (\ref{eq:A2}, $E_{\theta_0}$ stands for the expectation with respect to the true model.)
\end{itemize}

\begin{itemize}
\item[(A4)] 
Let $I=\left\{\theta:h(\theta)=\infty\right\}$. 
The prior $\pi$ satisfies $\pi(I)<1$.
\end{itemize}

Following the notation of \ctn{Shalizi09}, for $A\subseteq\Theta$, let
\begin{align}
h\left(A\right)&=\underset{\theta\in A}{\mbox{ess~inf}}~h(\theta);\notag\\
J(\theta)&=h(\theta)-h(\Theta);\notag\\
J(A)&=\underset{\theta\in A}{\mbox{ess~inf}}~J(\theta).\notag
\end{align}
\begin{itemize}
\item[(A5)] There exists a sequence of sets $\mathcal G_T\rightarrow\Theta$ as $T\rightarrow\infty$ 
such that: 
\begin{enumerate}
\item[(1)]
\begin{equation}
\pi\left(\mathcal G_T\right)\geq 1-\alpha\exp\left(-\beta (b_T-a_T)\right),~\mbox{for some}~\alpha>0,~\beta>2h(\Theta);
\label{eq:A5_1}
\end{equation}
\item[(2)]The convergence in (A2) is uniform in $\theta$ over $\mathcal G_T\setminus I$.
\item[(3)] $h\left(\mathcal G_T\right)\rightarrow h\left(\Theta\right)$, as $T\rightarrow\infty$.
\end{enumerate}
\end{itemize}
For each measurable $A\subseteq\Theta$, for every $\delta>0$, there exists a random natural number $\tau(A,\delta)$
such that
\begin{equation*}
(b_T-a_T)^{-1}\log\int_{A}R_T(\theta)\pi(\theta)d\theta
\leq \delta+\underset{(b_T-a_T)\rightarrow\infty}{\lim\sup}~(b_T-a_T)^{-1}
\log\int_{A}R_T(\theta)\pi(\theta)d\theta,
\end{equation*}
for all $(b_T-a_T)>\tau(A,\delta)$, provided 
$\underset{(b_T-a_T)\rightarrow\infty}{\lim\sup}~(b_T-a_T)^{-1}\log\pi\left(\mathbb I_A R_T\right)<\infty$.
Regarding this, the following assumption has been made by Shalizi:
\begin{itemize}
\item[(A6)] The sets $\mathcal G_T$ of (A5) can be chosen such that for every $\delta>0$, the inequality
$(b_T-a_T)>\tau(\mathcal G_T,\delta)$ holds almost surely for all sufficiently large $T$.
\end{itemize}
\begin{itemize}
\item[(A7)] The sets $\mathcal G_T$ of (A5) and (A6) can be chosen such that for any set $A$ with $\pi(A)>0$, 
\begin{equation*}
h\left(\mathcal G_T\cap A\right)\rightarrow h\left(A\right),
\end{equation*}
as $T\rightarrow\infty$.
\end{itemize}
Under the above assumptions, the following versions of the results of \ctn{Shalizi09} can be seen to hold.

\begin{theorem}[\ctn{Shalizi09}]
Consider assumptions (A1)--(A7) and any set $A\in\mathcal T$ with $\pi(A)>0$ and $h\left(A\right)>h\left(\Theta\right)$. 
Then almost surely,
\begin{equation*}
\underset{T\rightarrow\infty}{\lim}~\pi(A|\mathcal F_T)=0,
\end{equation*}
where $\pi(\cdot|\mathcal F_T)$ is the posterior distribution given 
$\mathcal F_T=\sigma\left(\left\{Y_s;~s\in[a_T,b_T]\right\}\right)$.
If $\beta>2h(A)$ or $A\subset\cap_{k=T}^{\infty}\mathcal G_k$, for some $T$, where
$\beta$ is given in (\ref{eq:A5_1}) under assumption (A5), then we further have, almost surely,
\begin{equation*}
\underset{T\rightarrow\infty}{\lim}~\frac{1}{b_T-a_T}\log\pi(A|\mathcal F_T)=-J(A).
\end{equation*}
\end{theorem}

\section{Proof of Lemma \ref{lemma:marginal_likelihood}}
\label{sec:proof_lemma_1}
We only need to verify that for any measurable and integrable function 
$g_T:\mathfrak Y_T\mapsto\mathbb R$,
$E_{T,X}\left[E_{T,Y|X}\left\{g_T(Y)\right\}\right]=E_{T,Y}\left[g(Y)\right]$,
where 
$\mathfrak Y_T$ denotes the sample space of $\left\{Y(t):~t\in[a_T,b_T]\right\}$,
$E_{T,X}$ denotes the marginal expectation with respect to the Girsanov formula based density
dominated by $Q_{T,X}$, $E_{T,Y|X}$ is the conditional
expectation with respect to the Girsanov formula based conditional density dominated by $Q_{T,Y|X}$, 
and $E_{T,Y}$ stands for the marginal expectation with respect
the proposed density $p_T(\theta_0)$ and the proposed dominating measure $Q_{T,Y}$. All the quantities
are associated with $[a_T,b_T]$.
Note that $E_{T,X}\left[E_{T,Y|X}\left\{|g_T(Y)|\right\}\right]<\infty$ if and only if
$E_{T,Y}\left[|g_T(Y)|\right]<\infty$, which easily follows from Tonelli's theorem related to
interchange of orders of integration for non-negative integrands.

Now, due to 
Fubini's theorem, for such integrable measurable function $g$,
\begin{align}
&E_{T,Y}\left[g_T(Y)\right]\notag\\
&=\int_{\mathfrak Y_T} g_T(y)p_T(\theta_0)dQ_{T,Y}\notag\\
&=\int_{\mathfrak Y_T} g_T(y)
\left[\int_{\mathfrak X_T}\exp\left(\phi_{Y,0} u_{Y|X,T}-\frac{\phi^2_{Y,0}}{2}v_{Y|X,T}\right)
\times \exp\left(\phi_{X,0} u_{X,T}-\frac{\phi^2_{X,0}}{2}v_{X,T}\right)dQ_{T,X}\right]dQ_{T,Y}\notag\\
& =\int_{\mathfrak X_T} 
\left[\int_{\mathfrak Y_T}\int_{\mathfrak X_T}g_T(y)\exp\left(\phi_{Y,0} u_{Y|X,T}
-\frac{\phi^2_{Y,0}}{2}v_{Y|X,T}\right)\right.\notag\\
&\qquad\qquad\qquad\left.
\times \exp\left(\phi_{X,0} u_{X,T}-\frac{\phi^2_{X,0}}{2}v_{X,T}\right)dQ_{T,X}dQ_{T,Y|X}\right]dQ_{T,X}\notag\\
&=\int_{\mathfrak X_T} 
\left[\int_{\mathfrak X_T}\left\{\int_{\mathfrak Y_T}g_T(y)\exp\left(\phi_{Y,0} u_{Y|X,T}-\frac{\phi^2_{Y,0}}{2}v_{Y|X,T}\right)
dQ_{T,Y|X}\right\}\right.\notag\\
&\qquad\qquad\qquad\left.
\times \exp\left(\phi_{X,0} u_{X,T}-\frac{\phi^2_{X,0}}{2}v_{X,T}\right)dQ_{T,X}\right]dQ_{T,X}\notag\\
&=\int_{\mathfrak X_T}
\left[\int_{\mathfrak X_T}E_{T,Y|X}\left\{g_T(Y)\right\}
\times \exp\left(\phi_{X,0} u_{X,T}-\frac{\phi^2_{X,0}}{2}v_{X,T}\right)
dQ_{T,X}\right]dQ_{T,X}\notag\\
&= \int_{\mathfrak X_T}E_{T,X}\left[E_{T,Y|X}\left\{g_T(Y)\right\}\right]dQ_{T,X}\notag\\
&=E_{T,X}\left[E_{T,Y|X}\left\{g_T(Y)\right\}\right],
\label{eq:marginal2}
\end{align}
since $\int_{\mathfrak X_T}dQ_{T,X}=1$ as $Q_{T,X}$ is a probability measure.
In particular, letting $g_T(y)=\mathbb I_{\mathfrak Y_T}(y)$, where for any set $A$, $\mathbb I_A$ denotes the
indicator function of $A$, the right hand side of (\ref{eq:marginal2}) becomes $1$, showing that
$p_T(\theta_0)$ is the correct density with respect to $Q_{T,Y}$.

\section{Proof of Lemma \ref{lemma:asymp_equal}}
\label{sec:proof_lemma_2}

Since the proofs of 
(\ref{eq:exp_I_Y_as}) and (\ref{eq:exp_I_X_as}) 
are the same, we provide the proof
of 
(\ref{eq:exp_I_Y_as}) only.

Consider the sequence 
$\delta_T\downarrow 0$ introduced in (H9). Then due to continuity of the exponential function,
\begin{align}
&P\left(\left|\exp\left(I_{Y,X,T}\right)
-\exp\left(\sqrt{K_Y}\left(W_Y(b_T)-W_Y(a_T)\right)\right)\right|>\epsilon\right)\notag\\
&\qquad\leq P\left(\left|I_{Y,X,T}-\sqrt{K_Y}\left(W_Y(b_T)-W_Y(a_T)\right)\right|>\delta_T\right)\label{eq:exp_as1}\\
&\quad\qquad+P\left(\left|I_{Y,X,T}-\sqrt{K_Y}\left(W_Y(b_T)-W_Y(a_T)\right)\right|\leq \delta_T,\right.\notag\\
&\qquad\qquad\left.~\left|\exp\left(I_{Y,X,T}\right)
-\exp\left(\sqrt{K_Y}\left(W_Y(b_T)-W_Y(a_T)\right)\right)\right|>\epsilon\right).
\label{eq:exp_as2}
\end{align}
The choice of $\delta_T\downarrow 0$ guarantees via (H9) that the terms (\ref{eq:exp_as2})
yield a convergent sum.

We now turn attention to $P\left(\left|I_{Y,X,T}-\sqrt{K_Y}\left(W_Y(b_T)-W_Y(a_T)\right)\right|>\delta_T\right)$.
Note that, almost surely, it holds due to (H3) and 
(\ref{eq:ss1}), 
that 
\begin{align}
\left|\frac{b_{Y}(Y(s),X(s),s)}{\sigma_{Y}(Y(s),X(s),s)}-\sqrt{K_Y}\right|
&\leq \left|\frac{b_{Y}(Y(s),X(s),s)}{\sigma_{Y}(Y(s),X(s),s)}\right|+\sqrt{K_Y}\notag\\
&\leq\sqrt{K_{Y,2,T}}\sqrt{1+\alpha_{Y,2,T}\xi^2\lambda^2(s)}+\sqrt{K_Y}\notag\\
&\leq 2\max\left\{\sqrt{K_{Y,2,T}},\sqrt{K_Y}\right\}\sqrt{1+\alpha_{Y,2,T}\xi^2\lambda^2(s)}.
\label{eq:bound_for_normal}
\end{align}
Now, noting the fact that for $k_0\geq 1$, $\lambda^{2k_0}(s)\leq \lambda^2(s)$ since $\lambda(t)\rightarrow 0$ as $t\rightarrow\infty$ and (H9), it holds due to (\ref{eq:bound_for_normal}), that 
\begin{align}
& \delta^{-2k_0}_T \left(b_T-a_T\right)^{k_0-1}
E\int_{a_T}^{b_T}\left|\frac{b_{Y}(Y(s),X(s),s)}{\sigma_{Y}(Y(s),X(s),s)}-\sqrt{K_Y}\right|^{2k_0}ds\notag\\
&\qquad\leq 2^{2k_0}\max\left\{K_{Y,2,T}^{k_0},K^{k_0}_Y\right\}
\delta^{-2k_0}_T \left(b_T-a_T\right)^{k_0-1}
E\int_{a_T}^{b_T}\left(1+\alpha_{Y,2,T}\xi^2\lambda^2(s)\right)^{k_0}ds
\label{eq:bound0}\\
&\qquad<\infty. ~~\mbox{(due to (H9))}\label{eq:bound1}
\end{align}

Due to (\ref{eq:bound1}) it follows that (see, Theorem 7.1 of \ctn{Mao11}, page 39)
\begin{align}
& \delta^{-2k_0}_T \left(b_T-a_T\right)^{k_0-1}
E\left|\int_{a_T}^{b_T}\left[\frac{b_Y(Y(s),X(s),s)}{\sigma_Y(Y(s),X(s),s)}-\sqrt{K_Y}\right]dW_Y(s)\right|^{2k_0}\notag\\
&\qquad\leq\left(k_0(2k_0-1)\right)^{k_0}
\delta^{-2k_0}_T \left(b_T-a_T\right)^{k_0-1}
E\int_{a_T}^{b_T}\left|\frac{b_Y(Y(s),X(s),s)}{\sigma_Y(Y(s),X(s),s)}-\sqrt{K_Y}\right|^{2k_0}ds.
\label{eq:for_chebychev}
\end{align}
Hence, using Chebychev's inequality, it follows using (\ref{eq:for_chebychev}) that for any $\epsilon>0$,
\begin{align}
&P\left(\left|\int_{a_T}^{b_T}
\left[\frac{b_Y(Y(s),X(s),s)}{\sigma_Y(Y(s),X(s),s)}-\sqrt{K_Y}\right]dW_Y(s)\right|>\delta_T\right)\notag\\
&\qquad<\left(k_0(2k_0-1)\right)^{k_0}\delta^{-2k_0}_T \left(b_T-a_T\right)^{k_0-1}
E\int_{a_T}^{b_T}\left|\frac{b_Y(Y(s),X(s),s)}{\sigma_Y(Y(s),X(s),s)}-\sqrt{K_Y}\right|^{2k_0}ds.
\end{align}
Using (\ref{eq:bound0}) and (H9), it follows that
\begin{equation}
\sum_{T=1}^{\infty}P\left(\left|\int_{a_T}^{b_T}\left[\frac{b_Y(Y(s),X(s),s)}
{\sigma_Y(Y(s),X(s),s)}-\sqrt{K_Y}\right]dW_Y(s)\right|>\delta_T\right)
<\infty.
\label{eq:exp_as4}
\end{equation}
Combining 
(\ref{eq:exp_as3}) of (H9) and (\ref{eq:exp_as4}) it follows that for all $\epsilon>0$,
\begin{equation*}
\sum_{T=1}^{\infty}P\left(\left|\exp\left(I_{Y,X,T}\right)
-\exp\left(\sqrt{K_Y}\left(W_Y(b_T)-W_Y(a_T)\right)\right)\right|>\epsilon\right)
<\infty,
\end{equation*}
proving that
$$\exp\left(I_{Y,X,T}\right)-\exp\left(\sqrt{K_Y}\left(W_Y(b_T)-W_Y(a_T)\right)\right)\stackrel{a.s.}{\longrightarrow} 0.$$

\section{Proof of Theorem \ref{theorem:asymp_approx}}
\label{sec:proof_theorem_asymp_approx}
Since the proofs of 
(\ref{eq:p_T_asymp}) and (\ref{eq:L_T_asymp}) 
are similar, we prove only 
(\ref{eq:p_T_asymp}).

By Lemma 
\ref{lemma:asymp_equal},
\begin{equation*}
\frac{\exp\left(I_{Y,X,T}\right)-\exp\left(\sqrt{K_Y}\left(W_Y(b_T)-W_Y(a_T)\right)\right)}
{\exp\left(\sqrt{K_Y}\left(W_Y(b_T)-W_Y(a_T)\right)\right)}
\stackrel{a.s.}{\longrightarrow} 0, 
\end{equation*}
so that
\begin{equation}
\frac{\exp\left(I_{Y,X,T}\right)}{\exp\left(\sqrt{K_Y}\left(W_Y(b_T)-W_Y(a_T)\right)\right)}\stackrel{a.s.}{\longrightarrow} 1.
\label{eq:asymp_approx2}
\end{equation}
Similarly,
\begin{equation}
\frac{\exp\left(I_{X,T}\right)}{\exp\left(\sqrt{K_X}\left(W_X(b_T)-W_X(a_T)\right)\right)}\stackrel{a.s.}{\longrightarrow} 1.
\label{eq:asymp_approx3}
\end{equation}
By (H3) and (H7), $\left((b_T-a_T)|K_{Y,j,T}-K_Y|\right)\rightarrow 0$ and 
$\left((b_T-a_T)|K_{X,j,T}-K_X|\right)\rightarrow 0$,
for $j=1,2$. 
Hence, it also holds that $\left((b_T-a_T)|K_{Y,1,T}-K_{Y,2,T}|\right)\rightarrow 0$
and $\left((b_T-a_T)|K_{X,1,T}-K_{X,2,T}|\right)\rightarrow 0$.
Also, by (H9), $\int_{a_T}^{b_T}\lambda^2(s)ds\rightarrow 0$.
Hence, it follows, as $\xi$ is a finite random variable, that
\begin{align}
&\frac{\exp\left(\phi^2_{Y,0}K_{Y,\xi,T,2}-\frac{\phi^2_{Y,0}K_{Y,\xi,T,1}}{2}\right)}
{\exp\left(\frac{(b_T-a_T)K_Y\phi^2_{Y,0}}{2}\right)}\notag\\
&=\exp\left[\frac{\phi^2_{Y,0}}{2}(b_T-a_T)(K_{Y,2,T}-K_Y)+\frac{\phi^2_{Y,0}}{2}(b_T-a_T)(K_{Y,2,T}-K_{Y,1,T})\right.\notag\\
&\qquad\qquad\left.+\frac{\phi^2_{Y,0}}{2}
\left(K_{Y,2,T}\alpha_{Y,2}-K_{Y,1,T}\alpha_{Y,1}\right)\xi^2\int_{a^T}^{b_T}\lambda^2(s)ds\right]\notag\\
&\stackrel{a.s.}{\longrightarrow} 1,
\label{eq:asymp_approx4}
\end{align}
and, similarly,
\begin{align}
&\frac{\exp\left(\phi^2_{X,0}K_{X,\xi,T,2}-\frac{\phi^2_{X,0}K_{X,\xi,T,1}}{2}\right)}
{\exp\left(\frac{(b_T-a_T)K_X\phi^2_{X,0}}{2}\right)}
\stackrel{a.s.}{\longrightarrow} 1,
\label{eq:asymp_approx5}
\end{align}
as $T\rightarrow\infty$.

Now, let
\begin{align}
\hat Z_{T,\theta_0}(W_X)
&=\exp\left(\frac{(b_T-a_T)K_Y\phi^2_{Y,0}}{2}+\phi_{Y,0}\sqrt{K_Y}\left(W_Y(b_T)-W_Y(a_T)\right)\right)\notag\\
&\qquad\qquad\times\exp\left(\frac{(b_T-a_T)K_X\phi^2_{X,0}}{2}+\phi_{X,0}\sqrt{K_X}\left(W_X(b_T)-W_X(a_T)\right)\right).\notag
\end{align}
From (\ref{eq:asymp_approx2}), (\ref{eq:asymp_approx3}), (\ref{eq:asymp_approx4}) and (\ref{eq:asymp_approx5})
it follows that $Z_{U,T,\theta_0}(X)/\hat Z_{T,\theta_0}(W_X)\rightarrow 1$, almost surely with respect to $W_X$ and $X$, 
as $T\rightarrow\infty$, given any fixed $W_Y$ in the respective non-null set.
That is, given any sequences $\left\{Z_{U,T,\theta_0}(X):~T>0\right\}$ and 
$\left\{\hat Z_{T,\theta_0}(W_X):~T>0\right\}$ associated with the complement
of null sets, 
for any $\epsilon>0$, there exists $T_0(\epsilon,W_Y,W_X)>0$ such that for $T\geq T_0(\epsilon,W_Y,W_X)$,
\begin{equation}
(1-\epsilon)\hat Z_{T,\theta_0}(W_X)\leq Z_{U,T,\theta_0}(X)\leq (1+\epsilon)\hat Z_{T,\theta_0}(W_X),
\end{equation}
for almost all $X$.
Thus, letting $g_{T,\theta_0}(W_X)=E\left[Z_{U,T,\theta_0}(X)|W_X\right]$, it follows that 
$g_{T,\theta_0}(W_X)-\hat Z_{T,\theta_0}(W_X)\stackrel{a.s.}{\longrightarrow}0$.
In fact, $$\frac{g_{T,\theta_0}(W_X)-\hat Z_{T,\theta_0}(W_X)}{\hat p_T(\theta_0)}\stackrel{a.s.}{\longrightarrow}0.$$
By (H10),
$$\underset{T>0}{\sup}~E\left(\frac{\left|Z_{U,T,\theta_0}(X)-\hat Z_{T,\theta_0}(W_X)\right|}{\hat p_T(\theta_0)}\right)
\leq\underset{T>0}{\sup}~E\left(\frac{Z_{U,T,\theta_0}(X)}{\hat p_T(\theta_0)}\right)+1<\infty.$$
Minor modification of Lemma B.119 of \ctn{Schervish95} then guarantee that 
$\frac{g_{T,\theta_0}(W_X)-\hat Z_{T,\theta_0}(W_X)}{\hat p_T(\theta_0)}$ is uniformly integrable.
Hence,
$$E\left(\frac{g_{T,\theta_0}(W_X)-\hat Z_{T,\theta_0}(W_X)}{\hat p_T(\theta_0)}\right)\rightarrow 0,$$
so that
$$E\left(\frac{g_{T,\theta_0}(W_X)}{\hat p_T(\theta_0)}\right)\rightarrow 1,~
\mbox{as}~ T\rightarrow\infty.$$
In other words, for almost all $W_Y$,
\begin{equation}
\frac{B_{U,T}(\theta_0)}{\hat p_T(\theta_0)}\rightarrow 1,~
\mbox{as}~ T\rightarrow\infty.
\label{eq:B_U_T_asymp}
\end{equation}
In the same way it follows that for almost all $W_Y$,
\begin{equation}
\frac{B_{L,T}(\theta_0)}{\hat p_T(\theta_0)}\rightarrow 1,~
\mbox{as}~ T\rightarrow\infty.
\label{eq:B_L_T_asymp}
\end{equation}
Combining (\ref{eq:B_U_T_asymp}) and (\ref{eq:B_L_T_asymp}) it follows that
$p_T(\theta_0)\sim\hat p_T(\theta_0)$ for almost all $W_Y$.

\section{Asymptotics in random effects models based on state space $SDE$s}
\label{sec:random_effects}

We make the following extra assumptions for investigating the asymptotic theory associated with (\ref{eq:data_sde3}), (\ref{eq:state_space_sde3}), (\ref{eq:data_sde4})
and (\ref{eq:state_space_sde4}).
\begin{itemize}
\item[(H15)]
For every $\theta\in\Theta\cup\{\theta_0\}$, $\psi_{Y_i}(\theta)$ and $\psi_{X_i}(\theta)$ are finite for 
all $i=1,\ldots,n$. And
\begin{align}
\psi_{Y_i}(\theta)&\rightarrow \bar\psi_Y(\theta);\notag\\
\psi_{X_i}(\theta)&\rightarrow \bar\psi_X(\theta),\notag
\end{align}
as $i\rightarrow\infty$, for all $\theta\in\Theta$. 
Also,
\begin{align}
K_{Y,i}&\rightarrow \bar K_Y;\notag\\
K_{X,i}&\rightarrow \bar K_X,\notag
\end{align}
as $i\rightarrow\infty$.
\item[(H16)] Assume that $\bar\psi_Y$ and $\bar\psi_X$ satisfy (H11) (i) -- (v).
\end{itemize}

\subsection{True likelihood}
\label{subsec:state_space_likelihood_bounds}

Here the true likelihood on $[a_{T},b_{T}]$ is given by 
\begin{equation}
\bar p_{n,T}(\theta_0)=\prod_{i=1}^np_{T,i}(\theta_0), 
\label{eq:true_likelihood_ss}
\end{equation}
where
\begin{align}
 p_{T,i}(\theta_0)&=\int \exp\left(\phi_{Y_i,0} u_{Y_i|X_i,T}-\frac{\phi^2_{Y_i,0}}{2}v_{Y_i|X_i,T}\right)
\times \exp\left(\phi_{X_i,0} u_{X_i,T}-\frac{\phi^2_{X_i,0}}{2}v_{X_i,T}\right)dQ_{T,X_i}.\notag 
\end{align}

It follows as in Section \ref{subsubsec:bounds_L_T} that
$\bar p_{n,T}(\theta_0)\stackrel{a.s.}{\sim}\prod_{i=1}^n\hat p_{T,i}(\theta_0)$, where
\begin{equation*}
\hat p_{T,i}(\theta_0)
=\exp\left(\frac{(b_{T}-a_{T})K_{Y_i}\phi^2_{Y_i,0}}{2}+\phi_{Y_i,0}\sqrt{K_{Y_i}}
\left(W_{Y_i}(b_{T})-W_{Y_i}(a_{T})\right)
+(b_{T}-a_{T})K_{X_i}\phi^2_{{X_i},0}\right).
\end{equation*}

\subsection{Modeled likelihood}
\label{subsec:modeled_likelihood_ss}
The modeled likelihood in this setup is given by
$\bar L_{n,T}(\theta)=\prod_{i=1}^nL_{T,i}(\theta)$, where
\begin{align}
L_{T,i}(\theta)&=\int \exp\left(\phi_{Y_i} u_{Y_i|X_i,T}-\frac{\phi^2_{Y_i}}{2}v_{Y_i|X_i,T}\right)
\times \exp\left(\phi_{X_i} u_{X_i,T}-\frac{\phi^2_{X_i}}{2}v_{X_i,T}\right)dQ_{T,X_i}. \notag 
\end{align}

As in Section \ref{subsec:modeled_likelihood} here it holds that 
$\bar L_{n,T}(\theta)\stackrel{a.s.}{\sim}\prod_{i=1}^n\hat L_{T,i}(\theta)$,
where
\begin{align}
\hat L_{T,i}(\theta)
&=\exp\left((b_T-a_T)K_{Y_i}\phi_{Y_i}\phi_{Y_i,0}+\phi_{Y_i}\sqrt{K_{Y_i}}\left(W_{Y_i}(b_T)-W_{Y_i}(a_T)\right)\right.\notag\\
&\qquad\qquad\left.-\frac{(b_T-a_T)K_{Y_i}\phi^2_{Y_i}}{2}+(b_T-a_T)K_{X_i}\phi_{X_i}\phi_{X_i,0}\right).\notag
\end{align}

\subsection{Bayesian consistency}
\label{subsec:bayesian_consistency_ss}

We now proceed to verify the assumptions of \ctn{Shalizi09}. First note that $L_{T,i}(\theta)$ is measurable
with respect to $\mathcal F_{T,i}\times \mathcal T$, where 
$\mathcal F_{T,i}=\sigma\left(\left\{Y_{is};s\in [a_T,b_T]\right\}\right)$,
the smallest $\sigma$-algebra with respect to which $\left\{Y_{is};s\in [a_T,b_T]\right\}$ is measurable. 
Let $\bar{\mathcal F}_{n,T}=\sigma\left(\left\{Y_{is};i=1,\ldots,n;~s\in [a_T,b_T]\right\}\right)$. 
Then for each $i=1,\ldots,n$, $L_{T,i}(\theta)$ is also $\bar{\mathcal F}_{n,T}\times \mathcal T$-measurable. It follows that
the likelihood
$\bar L_{n,T}(\theta)=\prod_{i=1}^nL_{T,i}(\theta)$
is measurable with respect to $\bar{\mathcal F}_{n,T}\times \mathcal T$. Hence, (A1) holds. 

Let $\bar R_{n,T}(\theta)=\frac{\bar L_{n,T}(\theta)}{\bar p_{n,T}(\theta_0)}$.
Then
\begin{equation*}
\frac{1}{n(b_T-a_T)}\log \bar R_{n,T}(\theta)=\frac{1}{n(b_T-a_T)}\sum_{i=1}^n\log R_{T,i},
\end{equation*}
where 
\begin{equation*}
R_{T,i}=\frac{L_{T,i}(\theta)}{p_{T,i}(\theta_0)}.
\end{equation*}

Since 
\begin{align}
\frac{1}{b_T-a_T}\log R_{T,i}(\theta)&\rightarrow 
-\frac{1}{2}\left[K_{Y,i}\left(\psi_{Y,i}(\theta)-\psi_{Y,i}(\theta_0)\right)^2+
K_{X,i}\left(\psi_{X,i}(\theta)-\psi_{X,i}(\theta_0)\right)^2\right.\notag\\
&\qquad\qquad\left.+K_X\left(\psi^2_{X,i}(\theta_0)-\psi^2_{X,i}(\theta)\right)\right]
\notag
\end{align}
for each $i$ as $T\rightarrow\infty$, it follows, using (H15), that
\begin{equation*}
\underset{n\rightarrow\infty}{\lim}\underset{T\rightarrow\infty}{\lim}
~\frac{1}{n(b_T-a_T)}\log \bar R_{n,T}(\theta) = -\bar h(\theta),
\end{equation*}
almost surely, where
\begin{equation*}
\bar h(\theta) = \frac{1}{2}\left[\bar K_{Y}\left(\bar\psi_{Y}(\theta)-\bar\psi_{Y}(\theta_0)\right)^2+
\bar K_{X}\left(\bar\psi_{X}(\theta)-\bar\psi_{X}(\theta_0)\right)^2
+\bar K_{X}\left(\bar\psi^2_{X}(\theta_0)-\bar\psi^2_{X}(\theta)\right)\right].
\end{equation*}
Thus, (A2) holds, and 
noting that $E\left(W_{Y_i}(b_T)-W_{Y,i}(a_T)\right)=0$,
it is easy to see that (A3) also holds.

We define, in our current context, the following:
\begin{align}
\bar h\left(A\right)&=\underset{\theta\in A}{\mbox{ess~inf}}~\bar h(\theta);\label{eq:h2_ss}\\
\bar J(\theta)&=\bar h(\theta)-\bar h(\Theta);\label{eq:J_ss}\\
\bar J(A)&=\underset{\theta\in A}{\mbox{ess~inf}}~\bar J(\theta).\label{eq:J2_ss}
\end{align}
The way of verification of (A4) remains the same as in Section \ref{subsubsec:A4}, with 
$I=\left\{\theta:\bar h(\theta)=\infty\right\}$. 
To verify (A5) (i) we define $\bar{\mathcal G}_{n,T}=\left\{\theta:\left|\bar\psi_Y(\theta)\right|
\leq\exp\left(\bar\beta n(b_T-a_T)\right)\right\}$, where $\bar\beta>2\bar h\left(\Theta\right)$. Coerciveness
of $\bar\psi_Y$ ensures compactness of $\bar{\mathcal G}_{n,T}$, and clearly, $\bar{\mathcal G}_{n,T}\rightarrow\Theta$,
as $n,T\rightarrow\infty$. Moreover,
$$\pi\left(\bar{\mathcal G}_{n,T}\right)>1-\bar\alpha\exp\left(-\bar\beta n(b_T-a_T)\right),$$
where $0<\bar\alpha=E\left(\left|\bar\psi_Y(\theta)\right|\right)<\infty$.
Verification of (A5) (ii) follows in the same way as in Section \ref{subsubsec:A5_2}, assuming (H10) 
holds for every $i$, and (A5) (iii) holds in the same way as in Section \ref{subsubsec:A5_3} with
$h$ replaced with $\bar h$ and $\mathcal G_T$ replaced with $\bar{\mathcal G}_{n,T}$.
Similarly as in Section \ref{subsubsec:A6} (A6) holds by additionally replacing
$R_T$ and $R_m$ with $\bar R_{n,T}$ and $\bar R_{n,m}$, respectively. Now, here
 $Z_m=\frac{1}{n(b_m-a_m)}\log\bar R_{n,m}(\hat\theta_T)+\bar h(\hat\theta_T)$ and
$\tilde Z_m=\sqrt{\bar K_Y}\frac{\left(\bar\psi_Y(\hat\theta_T)-\bar\psi_Y(\theta_0)\right)}{n\sqrt{b_m-a_m}}
\times\sum_{i=1}^n\frac{ W_{Y_i}(b_m)-W_{Y_i}(a_m)}{\sqrt{b_m-a_m}}$.

Note that
\begin{align}
\frac{\tilde Z^{8}_m}{E\left(\tilde Z^{8}_m\right)}=\frac{\left(\sum_{i=1}^n\frac{ W_{Y_i}(b_m)-W_{Y_i}(a_m)}{\sqrt{b_m-a_m}}\right)^8}{E\left(\left[\sum_{i=1}^n\frac{ W_{Y_i}(b_m)-W_{Y_i}(a_m)}{\sqrt{b_m-a_m}}\right]^6\right)}=\frac{\left(\sum_{i=1}^n\frac{W_{Y_i}(b_m)-W_{Y_i}(a_m)}{\sqrt{n(b_m-a_m)}}\right)^8}{E\left(\left[\sum_{i=1}^n\frac{ W_{Y_i}(b_m)-W_{Y_i}(a_m)}{\sqrt{n(b_m-a_m)}}\right]^8\right)}.\notag
\end{align}

Hence, even in this case,
\begin{align}
\frac{Z^{8}_m-\tilde Z^{8}_m}{E\left(\tilde Z^{8}_m\right)}
&=\frac{Z^{8}_m-\tilde Z^{8}_m}{\tilde Z^{8}_m}
\times\frac{\tilde Z^{8}_m}{E\left(\tilde Z^{8}_m\right)}
\stackrel{a.s.}{\longrightarrow}0~~\mbox{as}~~m\rightarrow\infty,
\label{eq:A6_33}
\end{align}
where the first factor on the right hand side of (\ref{eq:A6_33})
tends to zero almost surely as in Section \ref{subsubsec:A6}, while by the fact that 
$\frac{1}{\sqrt n}\sum_{i=1}^n\frac{W_{Y_i}(b_T)-W_{Y_i}(b_T)}{\sqrt{b_T-a_T}}\sim N(0,1)$, the second factor is bounded above 
by a constant times standard normal distribution raised to the power $6$. 
The rest of the verification is the same as in Section \ref{subsubsec:A6}.
It is also easy to see that (A7) holds, as in Section \ref{subsubsec:A7}.

We summarize our results in the form of the following theorem.
\begin{theorem}
\label{theorem:bayesian_convergence_ss}
Let the true, data-generating model be given by (\ref{eq:data_sde3}) and (\ref{eq:state_space_sde3}), but let
the data be modeled by (\ref{eq:data_sde4}) and (\ref{eq:state_space_sde4}).
Assume that (H1)--(H10) hold (for each $i=1,\ldots,n$, whenever appropriate); also assume (H13) -- (H16). 
Consider any set $A\in\mathcal T$ with $\pi(A)>0$ and $\bar h\left(A\right)>\bar h\left(\Theta\right)$. Then almost surely,
\begin{equation*}
\underset{n\rightarrow\infty,~T\rightarrow\infty}{\lim}~\pi(A|\bar{\mathcal F}_{n,T})=0.
\end{equation*}
Moreover, if $\bar\beta>2\bar h(A)$, then almost surely,  
\begin{equation*}
\underset{n\rightarrow\infty,~T\rightarrow\infty}{\lim}~\frac{1}{n(b_T-a_T)}\log\pi(A|\bar{\mathcal F}_{n,T})=-\bar J(A).
\end{equation*}
\end{theorem}

\subsection{Strong consistency and asymptotic normality of the maximum likelihood estimator of $\theta$}
\label{subsec:mle_strong_consistency_ss}

We now replace (H16) with 
\begin{itemize}
\item[(H16$^\prime$)] Assume that $\bar\psi_Y$ and $\bar\psi_X$ satisfy (H11) (i) -- (vii).
\end{itemize}
Let 
\begin{align}
\bar g_{Y,T}(\theta)&=-\frac{\bar K_Y}{2}\left(\bar\psi_Y(\theta)-\bar\psi_Y(\theta_0)\right)^2
+\sqrt{\bar K_Y}\left(\bar\psi_Y(\theta)-\bar\psi_Y(\theta_0)\right)\frac{1}{n}
\sum_{i=1}^n\frac{W_{Y_i}(b_T)-W_{Y_i}(a_T)}{b_T-a_T};\label{eq:g_Y_ss}\\
\bar g_{X,T}(\theta)&=-\frac{\bar K_X}{2}\left(\bar\psi_X(\theta)-\bar\psi_X(\theta_0)\right)^2
+\frac{\bar K_X}{2}\left(\bar\psi^2_X(\theta)-\bar\psi^2_X(\theta_0)\right).\notag 
\end{align}

Then note that
\begin{equation}
\underset{\theta\in\Theta}{\sup}~\left|\frac{1}{n(b_T-a_T)}\log \bar R_{n,T}(\theta)
-\bar g_{Y,T}(\theta)-\bar g_{X,T}(\theta)\right|
=\left|\frac{1}{n(b_T-a_T)}\log \bar R_{n,T}(\bar\theta^*_{nT})-\bar g_{Y,T}(\bar\theta^*_{nT})-\bar g_{X,T}(\bar\theta^*_{nT})\right|,
\label{eq:mle_unicon_ss}
\end{equation}
for some $\bar\theta^*_{nT}\in\Theta$. 
As before despite the dependence of $\bar\theta^*_{nT}$ on data it can be shown that 
(\ref{eq:mle_unicon_ss}) tends to zero as $T\rightarrow\infty$.
So, it is permissible to approximate the $MLE$ by maximizing 
$$\bar g_{n,T}(\theta)=\bar g_{Y,T}(\theta)+\bar g_{X,T}(\theta)$$ with respect to $\theta$.
%
%

Let $$\bar g'_{n,T}(\theta)=\left(\frac{\partial\bar g_{n,T}(\theta)}{\partial\theta_1},\ldots,
\frac{\partial\bar g_{n,T}(\theta)}{\partial\theta_d}\right)^T,$$
and let $\bar g''_{n,T}(\theta)$ be the matrix of second derivatives. The relevant elements at $\theta=\theta_0$ are given by 
\begin{align}
\left[\frac{\partial\bar g_{n,T}(\theta)}{\partial\theta_k}\right]_{\theta=\theta_0}
&=\sqrt{\bar K_Y}\left[\frac{\partial\bar\psi_Y(\theta)}{\partial\theta_k}\right]_{\theta=\theta_0}
\frac{1}{n}\sum_{i=1}^n\frac{W_{Y_i}(b_T)-W_{Y_i}(a_T)}{b_T-a_T};\notag\\
\left[\frac{\partial^2 \bar g_{n,T}(\theta)}{\partial\theta_j\partial\theta_k}\right]_{\theta=\theta_0}
&=-\bar K_Y\left[\frac{\partial\bar\psi_Y(\theta)}{\partial\theta_j}
\frac{\partial\bar\psi_Y(\theta)}{\partial\theta_k}\right]_{\theta=\theta_0}
+\sqrt{\bar K_Y}\left[\frac{\partial^2\bar\psi_Y(\theta)}{\partial\theta_j\partial\theta_k}\right]_{\theta=\theta_0}
 \frac{1}{n}\sum_{i=1}^n\frac{W_{Y_i}(b_T)-W_{Y_i}(a_T)}{b_T-a_T}.\notag
\end{align}

In this case, the $(j,k)$-th element of the matrix
${\mathcal I}(\theta_0)$ is given by
\begin{equation*}
\left\{{\mathcal I}(\theta_0)\right\}_{jk}
=\bar K_Y\left[\frac{\partial\bar\psi_Y(\theta)}{\partial\theta_j}\frac{\partial\bar\psi_Y(\theta)}
{\partial\theta_k}\right]_{\theta=\theta_0},
\end{equation*}
and the $MLE$ $\hat\theta_{n,T}$ satisfies
\begin{equation}
{\mathcal I}^{-1}(\theta^*_{n,T})\bar g''_{n,T}(\theta^*_{n,T})\left(\hat\theta_{n,T}-\theta_0\right)
=-{\mathcal I}^{-1}(\theta^*_{n,T})\bar g'_{n,T}(\theta_0),
\label{eq:mle1_ss}
\end{equation}
where $\theta^*_{n,T}$ lies between $\theta_0$ and $\hat\theta_{n,T}$. 
It is easily seen as in Section \ref{subsec:mle_strong_consistency} that
\begin{equation}
\hat\theta_{n,T}\stackrel{a.s.}{\longrightarrow}\theta_0,
\label{eq:mle_theta_ss}
\end{equation}
as $n\rightarrow\infty$, $T\rightarrow\infty$.
\begin{theorem}
\label{theorem:mle_strong_consistency_ss}
Let the true, data-generating model be given by (\ref{eq:data_sde3}) and (\ref{eq:state_space_sde3}), but let
the data be modeled by (\ref{eq:data_sde4}) and (\ref{eq:state_space_sde4}).
Assume that (H1)--(H10), (H12)--(H15) and (H16$^\prime$) hold (for each $i=1,\ldots,n$, whenever appropriate). 
Then the $MLE$ of $\theta$ is strongly consistent in the sense that (\ref{eq:mle_theta_ss}) holds.
\end{theorem}

Moreover, following the same ideas presented in Section \ref{subsec:mle_normality}, and employing 
(H15$^\prime$), it is easily seen that
asymptotic normality also holds. Formally, we have the following theorem.
\begin{theorem}
\label{theorem:mle_normality_ss}
Let the true, data-generating model be given by (\ref{eq:data_sde3}) and (\ref{eq:state_space_sde3}), but let
the data be modeled by (\ref{eq:data_sde4}) and (\ref{eq:state_space_sde4}).
Assume that (H1)--(H10),  (H12)--(H15) and (H16$^\prime$) hold (for each $i=1,\ldots,n$, whenever appropriate).  
Then
\begin{equation*}
\sqrt{n(b_T-a_T)}\left(\hat\theta_{n,T}-\theta_0\right)\stackrel{\mathcal L}{\longrightarrow}
N_d\left(0,{\mathcal I}^{-1}(\theta_0)\right),
\end{equation*}
as $n\rightarrow\infty$, $T\rightarrow\infty$.
\end{theorem}

\subsection{Asymptotic posterior normality}
\label{subsec:verification_posterior_normality_ss}

From Section \ref{subsec:mle_strong_consistency_ss} (see \ref{eq:mle_unicon_ss})
it is evident that $\frac{1}{n(b_T-a_T)}\ell_{n,T}(\theta)$, where $\ell_{n,T}(\theta)=\log L_{n,T}(\theta)$, 
can be uniformly approximated by 
\begin{equation*}
\frac{1}{n(b_T-a_T)}\bar\ell_{n,T}(\theta)=\bar g_{Y,T}(\theta)+\bar g_{X,T}(\theta)+\frac{1}{n(b_T-a_T)}\log \bar{p}_{nT}(\theta_0), 
\end{equation*}
for $\theta\in\Theta$.
%
With this approximate version, it is again easy to see that the first four
regularity conditions presented in Section \ref{subsec:regularity_conditions} trivially hold.

We now verify regularity condition (5).
Since, as $n\rightarrow\infty$, $T\rightarrow\infty$, $\hat\theta_{n,T}\stackrel{a.s.}{\longrightarrow}\theta_0$, 
\begin{equation*}
\frac{1}{n(b_T-a_T)}\bar\ell''_{n,T}(\hat\theta_{n,T})\stackrel{a.s.}{\longrightarrow}
-{\mathcal I}(\theta_0).
\end{equation*}
Thus, as before, almost surely, $$\bar\Sigma^{-1}_{n,T}\sim n(b_T-a_T)\times {\mathcal I}(\theta_0),$$
where
\begin{equation*}
\bar\Sigma^{-1}_{n,T}=\left\{\begin{array}{cc}
-\bar\ell''_{n,T}(\hat\theta_{n,T}) & \mbox{if the inverse and}\ \ \hat\theta_{n,T}\ \ \mbox{exist}\\
\mathfrak I_d & \mbox{if not},
\end{array}\right.
\end{equation*}
Hence, $$\bar\Sigma_{n,T}\stackrel{a.s.}{\longrightarrow} 0,$$ as $n\rightarrow\infty$, $T\rightarrow\infty$. 
%
Thus, regularity condition (5) holds.

For verifying condition (6), observe that
\begin{equation*}
\rho_{n,T}\left[\ell_{n,T}(\theta)-\ell_{n,T}(\theta_0)\right]
=\rho_{n,T}n(b_T-a_T)\times\frac{1}{n(b_T-a_T)}\log R_{n,T}(\theta),
\end{equation*}
where $\rho_{n,T}$ is the smallest eigenvalue of $\bar\Sigma_{n,T}$, and, as in Section \ref{subsec:mle_normality}. 
$\rho_{n,T}n(b_T-a_T)\rightarrow \bar c$, for some $\bar c>0$. 
Then, as in (\ref{eq:6_2}), it holds that
\begin{align}
\rho_{n,T}\left[\ell_{n,T}(\theta)-\ell_{n,T}(\theta_0)\right]&\stackrel{a.s.}{\longrightarrow}
-\frac{\bar c}{2}\left[\bar K_Y\left(\bar\psi_Y(\theta)-\bar\psi_Y(\theta_0)\right)^2
+\bar K_X\left(\bar\psi_X(\theta)-\bar\psi_X(\theta_0)\right)^2
\right.\notag\\
&\qquad\qquad\left.+\bar K_X\left(\bar\psi^2_X(\theta_0)-\bar\psi^2_X(\theta)\right)\right],\notag
\end{align}
for all $\theta\in \Theta\setminus\mathcal N_0(\delta)$.
Then, in the same way as in (\ref{eq:6_3}) it follows that
\begin{align}
&\underset{n\rightarrow\infty,~T\rightarrow\infty}{\lim}~
P_{\theta_0}\left(\underset{\theta\in\Theta\backslash\mathcal N_0(\delta)}{\sup}~
\rho_{n,T}\left[\ell_{n,T}(\theta)-\ell_{n,T}(\theta_0)\right]<-K(\delta)\right)=1.\notag
\end{align}
In other words, condition (6) holds.

Condition (7) can be verified essentially in the same way as in Section \ref{subsec:mle_normality}. 
As in Section \ref{subsec:mle_normality}, using continuity of the third derivatives of
$\bar\psi_Y$ and $\bar\psi_X$, as assumed in (H15$^\prime$) it can be shown
that $\bar{\ell}''_{n,T}(\theta)=O\left(-n(b_T-a_T)\times {\mathcal I}(\theta_0)+n(b_T-a_T)\delta_2\right)$,
almost surely.
It is also easy to see that $\bar\Sigma^{\frac{1}{2}}_{n,T}$ is asymptotically almost surely 
equivalent to $n^{-\frac{1}{2}}(b_T-a_T)^{-\frac{1}{2}}
{\mathcal I}^{-\frac{1}{2}}(\theta_0)$. Thus, condition (7) holds.

We summarize our result in the form of the following theorem.
\begin{theorem}
\label{theorem:theorem5_ss}
Let the true, data-generating model be given by (\ref{eq:data_sde3}) and (\ref{eq:state_space_sde3}), but let
the data be modeled by (\ref{eq:data_sde4}) and (\ref{eq:state_space_sde4}).
Assume that (H1) -- (H10), (H12)--(H15) and (H16$^\prime$) hold (for every $i=1,\ldots,n$, whenever appropriate). 
Then denoting $\bar\Psi_{n,T}=\bar\Sigma^{-1/2}_{n,T}\left(\theta-\hat\theta_{n,T}\right)$, for each compact subset 
$B$ of $\mathbb R^d$ and each $\epsilon>0$, the following holds:
\begin{equation*}
\lim_{n\rightarrow\infty,T\rightarrow\infty}P_{\theta_0}
\left(\sup_{\bar\Psi_{n,T}\in B}\left\vert\pi(\bar\Psi_{n,T}\vert\bar{\mathcal F}_{n,T})
-\varrho(\bar\Psi_T)\right\vert>\epsilon\right)=0.
\end{equation*}
\end{theorem}

\section{Asymptotic theory for multidimensional linear random effects}
\label{sec:multidimensional}

We now consider the following true, multidimensional linear random effects models based on state space $SDE$s:
for $i=1,\ldots,n$, and for $t\in[0,b_T]$,
\begin{align}
d Y_i(t)&=\phi^T_{Y_i,0} b_Y(Y_i(t),X_i(t),t)dt+\sigma_Y(Y_i(t),X_i(t),t)dW_{Y,i}(t);
\label{eq:data_sde3_mult}\\
d X_i(t)&=\phi^T_{X_i,0}b_X(X_i(t),t)dt+\sigma_X(X_i(t),t)dW_{X,i}(t).
\label{eq:state_space_sde3_mult}
\end{align}
In the above, $\phi_{Y_i,0}=\phi_{Y_i,0}(\theta_0)=\left(\psi_{Y_i,1}(\theta_0),\ldots,\psi_{Y_i,r_Y}(\theta_0)\right)^T$ 
and $\phi_{X_i,0}=\phi_{X_i,0}(\theta_0)=\left(\psi_{X_i,1}(\theta_0),\ldots,\psi_{X_i,r_X}(\theta_0)\right)^T$, where
$\left\{\psi_{Y_i,j};j=1,\ldots,r_Y\right\}$ and $\left\{\psi_{X_i,j};j=1,\ldots,r_X\right\}$ 
are known functions, $r_Y~(>1)$ and $r_X~(>1)$ are dimensions of the multivariate functions 
$\phi_{Y_i,0}$ and $\phi_{X_i,0}$; $\theta_0$ is the true set of parameters.
Also, $b_Y(y,x)=\left(b_{Y,1}(y,x),\ldots,b_{Y,r_Y}(y,x)\right)^T$ and 
$b_X(y,x)=\left(b_{X,1}(y,x),\ldots,b_{X,r_X}(y,x)\right)^T$
are $r_Y$ and $r_X$ dimensional functions respectively.

Our modeled state space $SDE$ is given, for $t\in [0,b_T]$ by:
\begin{align}
d Y_i(t)&=\phi^T_{Y_i} b_Y(Y_i(t),X_i(t),t)dt+\sigma_Y(Y_i(t),X_i(t),t)dW_{Y,i}(t);
\label{eq:data_sde4_mult}\\
d X_i(t)&=\phi^T_{X_i}b_X(X_i(t),t)dt+\sigma_X(X_i(t),t)dW_{X,i}(t),
\label{eq:state_space_sde4_mult}
\end{align}
where $\phi_{Y_i}=\phi_{Y_i}(\theta)=\left(\psi_{Y_i,1}(\theta),\ldots,\psi_{Y_i,r_Y}(\theta)\right)^T$ and 
$\phi_{X_i}=\phi_{X_i}(\theta)=\left(\psi_{X_i,1}(\theta),\ldots,\psi_{X_i,r_X}(\theta)\right)^T$. 
In this section we generalize our asymptotic theory in the case of the above multidimensional random effects models 
based on state space $SDE$s. 

Let $\tilde b_Y(y,x,\theta_0)=\phi^T_{Y,0}b_Y(y,x)$ and $\tilde b_Y(y,x,\theta)=\phi^T_{Y}b_Y(y,x)$. Also let
$\tilde b_X(x,\theta_0)=\phi^T_{X,0}b_X(x)$ and $\tilde b_X(x,\theta)=\phi^T_{X}b_X(x)$.
We assume that given any $\theta\in\Theta$, $\tilde b_{Y_i}$ and $\tilde b_{X_i}$ satisfy conditions 
(H1) -- (H7) of MB. However, we replace (H3) and (H7) with the following:
\begin{itemize}
\item[(H3$^\prime$)]  For every pair $(j_1,j_2);j_1=1,\ldots,r_Y;j_2=1,\ldots,r_Y$, and for 
every $T>0$, there exist positive constants 
$K_{Y,1,T,j_1,j_2}$, $K_{Y,2,T,j_1,j_2}$, $\alpha_{Y,1,j_1,j_2}$, $\alpha_{Y,2,j_1,j_2}$
such that for all $(x,t)\in\mathbb R\times [0,b_T]$,
$$K_{Y,1,T,j_1,j_2}\left(1-\alpha_{Y,1,j_1,j_2}x^2\right)
\leq\frac{b_{Y,j_1}(y,x,t)b_{Y,j_2}(y,x,t)}{\sigma^2_Y(y,x,t)}
\leq K_{Y,2,T,j_1,j_2}\left(1+\alpha_{Y,2,j_1,j_2}x^2\right),$$
where $K_{Y,1,T,j_1,j_2}\rightarrow K_{Y,j_1,j_2}$ and 
$K_{Y,2,T,j_1,j_2}\rightarrow K_{Y,j_1,j_2}$ as $T\rightarrow\infty$; $K_{Y,j_1,j_2}$ being  
positive constants. We further assume for $k=1,2, (b_T-a_T)|K_{Y,k,T,j_1,j_2}-K_{Y,j_1,j_2}|\rightarrow 0$, as $T\rightarrow\infty$.
\item[(H7$^\prime$)]  For every pair $(j_1,j_2);j_1=1,\ldots,r_X;j_2=1,\ldots,r_X$, and for every $T>0$, 
there exist positive constants 
$K_{X,1,T,j_1,j_2}$, $K_{X,2,T,j_1,j_2}$, $\alpha_{X,1,j_1,j_2}$, $\alpha_{X,2,j_1,j_2}$
such that for all $(x,t)\in\mathbb R\times [0,b_T]$,
$$K_{X,1,T,j_1,j_2}\left(1-\alpha_{X,1,j_1,j_2}x^2\right)
\leq\frac{b_{X,j_1}(x,t)b_{X,j_2}(x,t)}{\sigma^2_X(x,t)}
\leq K_{X,2,T,j_1,j_2}\left(1+\alpha_{X,2,j_1,j_2}x^2\right),$$
where $K_{X,1,T,j_1,j_2}\rightarrow K_{X,j_1,j_2}$ and $K_{X,2,T,j_1,j_2}\rightarrow K_{X,j_1,j_2}$, as $T\rightarrow\infty$;
$K_{X,j_1,j_2}$ being positive constants. We also assume for $k=1,2, (b_T-a_T)|K_{X,k,T,j_1,j_2}-K_{X,j_1,j_2}|\rightarrow 0$, as $T\rightarrow\infty$.
\end{itemize}

For each $i=1,\ldots,n$, we re-define $u_{Y_i|X_i,T}$ and $u_{X_i,T}$ as $r_Y$ and $r_X$ dimensional
vectors, with elements given by:
\begin{align}
u_{Y_i|X_i,T,j}&=\int_{a_T}^{b_T}\frac{b_{Y,j}(Y_i(s),X_i(s),s)}{\sigma^2_Y(Y_i(s),X_i(s),s)}dY_i(s);
~j=1,\ldots,r_Y;\notag\\ 
u_{X_i,T,j}&=\int_{a_T}^{b_T}\frac{b_{X,j}(X_i(s),s)}{\sigma^2_X(X_i(s),s)}dX_i(s);~j=1,\ldots,r_X.\notag 
\end{align}
Also, for $i=1,\ldots,n$, let us define $r_Y\times r_Y$ and $r_X\times r_X$ 
matrices $v_{Y_i|X_i,T}$ and $v_{X_i,T}$ with $(j_1,j_2)$-th elements
\begin{align}
v_{Y_i|X_i,T,j_1,j_2}
&=\int_{a_T}^{b_T}\frac{b_{Y,j_1}(Y_i(s),X_i(s),s)b_{Y,j_2}(Y_i(s),X_i(s),s)}{\sigma^2_Y(Y_i(s),X_i(s),s)}ds;
~j_1=1,\ldots,r_Y;~j_2=1,\ldots,r_Y;\notag\\ 
v_{X_i,T,j_1,j_2}&=\int_{a_T}^{b_T}\frac{b_{X,j_1}(X_i(s),s)b_{X,j_2}(X_i(s),s)}{\sigma^2_X(X_i(s),s)}ds;
~j_1=1,\ldots,r_X;~j_2=1,\ldots,r_X.\notag 
\end{align}

We assume that
\begin{itemize}
\item[(H17)] For $i=1,\ldots,n$, $v_{Y_i|X_i,T}$ and $v_{X_i,T}$ are positive definite matrices.
\end{itemize}
We also replace (H16) of MB with the following.
\begin{itemize}
\item[(H16$^{\prime\prime}$)]
For every $\theta\in\Theta\cup\{\theta_0\}$, $\psi_{Y_{i,j}}(\theta)$ and $\psi_{X_{i,j}}(\theta)$ 
are finite for all $i=1,\ldots,n; j=1,\ldots,r_Y; j=1,\ldots,r_X.$ And
\begin{align}
\psi_{Y_i,j}(\theta)&\rightarrow \bar\psi_Y(\theta,j);~j=1,\ldots,r_Y\notag\\
\psi_{X_i,j}(\theta)&\rightarrow \bar\psi_X(\theta,j);~j=1,\ldots,r_X,\notag
\end{align}
as $i\rightarrow\infty$, for all $\theta\in\Theta$, where 
$\bar\psi_{Y}(\theta)=\left(\bar\psi_Y(\theta,1),\ldots,\bar\psi_Y(\theta,r_Y)\right)^T$ and 
$\bar\psi_{X}(\theta)=\left(\bar\psi_X(\theta,1),\ldots,\bar\psi_X(\theta,r_X)\right)^T$ are
coercive functions with continuous third derivatives. 
Here, by coerciveness we mean $\left\|\bar\psi_{Y}(\theta)\right\|$ and $\left\|\bar\psi_X(\theta)\right\|$
tend to infinity as $\|\theta\|\rightarrow\infty$.
We additionally assume that $\left\|\bar\psi_Y(\theta)\right\|$ has 
finite expectation with respect to the prior $\pi(\theta)$.
Also, for $k=1,2$ letting $K_{Y,k,i}$ and $K_{X,k,i}$ denote matrices with $(j_1,j_2)$-th elements $K_{Y,k,i,j_1,j_2}$
and $K_{X,k,i,j_1,j_2}$ as in (H3$^\prime$) and (H7$^\prime$), we assume
\begin{align}
K_{Y,k,i}&\rightarrow \bar K_Y;\notag\\
K_{X,k,i}&\rightarrow \bar K_X,\notag
\end{align}
as $i\rightarrow\infty$, where $\bar K_Y$ and $\bar K_X$ are positive definite matrices,
and that $$\left(\bar\psi_{X}(\theta_0)-\bar\psi_{X}(\theta)\right)^T
\bar K_X\left(\bar\psi_{X}(\theta_0)+\bar\psi_{X}(\theta)\right)\geq 0$$ for all $\theta\in\Theta$. We assume that for every sequence $\{\theta_T:T>0\}$ such that $\parallel\theta_T\parallel\rightarrow\infty$, 
as $T\rightarrow\infty$,
\begin{enumerate}
\item[(i)]$(b_T-a_T)|(\phi_{Y_i}(\theta_T))^T(K_{Y,k,T,i}-K_{Y,k,i})(\phi_{Y_i}(\theta_T))|\rightarrow 0$, 
for every $i=1,2,\ldots$, and for $k=1,2$; 
\item[(ii)]$(b_T-a_T)|(\phi_{X_i}(\theta_T))^T(K_{X,k,T,i}-K_{X,k,i})(\phi_{X_i}(\theta_T))|\rightarrow 0$
for every $i=1,2,\ldots$, and for $k=1,2$; 
\item[(iii)]$C_1(b_T-a_T)\leq \parallel\bar\psi_Y(\theta_T)-\bar\psi_Y(\theta_0)\parallel^8
\leq C_2(b_T-a_T)$, for some $C_1,C_2>0$.
\end{enumerate}
\end{itemize}

\subsection{True and modeled likelihood in the multidimensional case}
\label{subsec:state_space_likelihood_bounds_mult}

Here the true likelihood is given by 
\begin{equation*}
\bar p_{n,T}(\theta_0)=\prod_{i=1}^np_{T,i}(\theta_0), 
\end{equation*}
where
\begin{align}
 p_{T,i}(\theta_0)&=\int \exp\left(\phi^T_{Y_i,0} u_{Y_i|X_i,T}
-\frac{1}{2}\phi^T_{Y_i,0} v_{Y_i|X_i,T}\phi_{Y_i,0}\right)
\times \exp\left(\phi^T_{X_i,0} u_{X_i,T}-\frac{1}{2}\phi^T_{X_i,0}v_{X_i,T}\phi_{X_i,0}\right)dQ_{T,X_i}.\notag
\end{align}
The modeled likelihood is given by
$\bar L_{n,T}(\theta)=\prod_{i=1}^nL_{T,i}(\theta)$, where
\begin{align}
L_{T,i}(\theta)&=\int \exp\left(\phi^T_{Y_i} u_{Y_i|X_i,T}-\frac{1}{2}\phi^T_{Y_i}v_{Y_i|X_i,T}\phi_{Y_i}\right)
\times \exp\left(\phi^T_{X_i} u_{X_i,T}-\frac{1}{2}\phi^T_{X_i}v_{X_i,T}\phi_{X_i}\right)dQ_{T,X_i}.\notag 
\end{align}

The inequalities (2.11) and (2.12)  of MB hold for $i=1,\ldots,n$, but now
$K_{Y,1,T}$ and $K_{Y,2,T}$ in (2.13) and (2.14) of MB
are matrices with $(j_1,j_2)$-th elements $K_{Y,1,T,j_1,j_2}$ and $K_{Y,2,T,j_1,j_2}$, respectively,
as described in (H7$^\prime$).

As before the likelihood $\bar L_{n,T}(\theta)=\prod_{i=1}^nL_{T,i}(\theta)$ is easily seen to be measurable, so that
(A1) of MB holds.

It is easily seen, as before, that
\begin{equation*}
\underset{n\rightarrow\infty}{\lim}\underset{T\rightarrow\infty}{\lim}
~\frac{1}{n(b_T-a_T)}\log \bar R_{n,T}(\theta) = -\bar h(\theta),
\end{equation*}
almost surely, where
\begin{align}
\bar h(\theta) &= \frac{1}{2}\left[\left(\bar\psi_{Y}(\theta)-\bar\psi_{Y}(\theta_0)\right)^T\bar K_Y
\left(\bar\psi_{Y}(\theta)-\bar\psi_{Y}(\theta_0)\right)+
\left(\bar\psi_{X}(\theta)-\bar\psi_{X}(\theta_0)\right)^T\bar K_X
\left(\bar\psi_{X}(\theta)-\bar\psi_{X}(\theta_0)\right)\right]\notag\\
&\qquad\qquad
+\left(\bar\psi_{X}(\theta_0)-\bar\psi_{X}(\theta)\right)^T\bar K_X\left(\bar\psi_{X}(\theta_0)+\bar\psi_{X}(\theta)\right).\notag
\end{align}
Thus, (A2) of MB holds, and as before, (A3) of MB is also clearly seen to hold, 

The way of verification of (A4) of MB remains the same as in Section 4.4 of MB, with 
$I=\left\{\theta:\bar h(\theta)=\infty\right\}$. 
Condition (A5) can be seen to hold as before, and defining \\
$\bar{\mathcal G}_{n,T}=\left\{\theta:\left\|\bar\psi_Y(\theta)\right\|
\leq\exp\left(\bar\beta n(b_T-a_T)\right)\right\}$, where $\bar\beta>2\bar h\left(\Theta\right)$ and
$\bar\alpha=E\left\|\bar\psi_Y(\theta)\right\|$, (A5) is verified as before.
That (A6) and (A7) of MB hold can be argued as before. 
We thus have the following theorem.
\begin{theorem}
\label{theorem:bayesian_convergence_ss_mult}
Assume that the data was generated by the true model given by (\ref{eq:data_sde3_mult}) 
and (\ref{eq:state_space_sde3_mult}), but
modeled by (\ref{eq:data_sde4_mult}) and (\ref{eq:state_space_sde4_mult}).
Assume that (H1)--(H10) of MB hold (for each $i=1,\ldots,n$, whenever appropriate), with (H3) and (H7) replaced with
(H3$^\prime$) and (H7$^\prime$), respectively. Also assume (H13)--(H15) of MB , (H16$^{\prime\prime}$) and (H17). 
Consider any set $A\in\mathcal T$ with $\pi(A)>0$ and $\bar h\left(A\right)>\bar h\left(\Theta\right)$. Then almost surely,
\begin{equation*}
\underset{n\rightarrow\infty,~T\rightarrow\infty}{\lim}~\pi(A|\bar{\mathcal F}_{n,T})=0.
\end{equation*}
Moreover, if $\bar\beta>2\bar h(A)$, then almost surely,
\begin{equation*}
\underset{n\rightarrow\infty,~T\rightarrow\infty}{\lim}~\frac{1}{n(b_T-a_T)}\log\pi(A|\bar{\mathcal F}_{n,T})=-\bar J(A).
\end{equation*}
\end{theorem}

\subsection{Strong consistency and asymptotic normality of the maximum likelihood estimator of $\theta$}
\label{subsec:mle_strong_consistency_ss_mult}

We now assume, in addition to (H16$^{\prime\prime}$) that
\begin{itemize}
\item[(H16$^{\prime\prime\prime}$)]
$\bar\psi_{Y}(\theta)$ and $\bar\psi_{X}(\theta)$ are thrice continuously differentiable
and that the first two derivatives of $\bar\psi_{X}(\theta)$ vanish at $\theta_0$.
\end{itemize}

In this case, the $MLE$ satisfies 
(\ref{eq:mle1_ss}) with the appropriate multivariate extension as detailed
above
where the $(j,k)$-th element of ${\mathcal I}(\theta)$
is given by
\begin{equation}
\left\{{\mathcal I}(\theta)\right\}_{jk}
=\left(\frac{\partial\bar\psi_Y(\theta)}{\partial\theta_j}\right)^T\bar K_Y
\left(\frac{\partial\bar\psi_Y(\theta)}{\partial\theta_k}\right).
\label{eq:I_ss_mult}
\end{equation}
Thus, as before, it can be shown that strong consistency of the $MLE$ of the form 
(\ref{eq:mle_theta_ss}) holds.
Formally, we have the following theorem.
\begin{theorem}
\label{theorem:mle_strong_consistency_ss_mult}
Assume that the data was generated by the true model given by (\ref{eq:data_sde3_mult}) and (\ref{eq:state_space_sde3_mult}), but
modeled by (\ref{eq:data_sde4_mult}) and (\ref{eq:state_space_sde4_mult}).
Assume that (H1)--(H10), (H12) -- (H15) of MB hold 
(for each $i=1,\ldots,n$, whenever appropriate), with (H3) and (H7) replaced with
(H3$^\prime$) and (H7$^\prime$), respectively. Also assume (H16$^{\prime\prime}$), 
(H16$^{\prime\prime\prime}$) and (H17). 
Then the $MLE$ is strongly consistent, that is,
\begin{equation*}
\hat\theta_{n,T}\stackrel{a.s.}{\longrightarrow}\theta_0,
\end{equation*}
as $n\rightarrow\infty$, $T\rightarrow\infty$.
\end{theorem}
Asymptotic normality of the form (7.28) of MB also holds, where the elements of the information
matrix ${\mathcal I}(\theta_0)$ are given by (\ref{eq:I_ss_mult}). 
Here the formal theorem is given as follows.
\begin{theorem}
\label{theorem:mle_normality_ss_mult}
Assume that the data was generated by the true model given by (\ref{eq:data_sde3_mult}) and (\ref{eq:state_space_sde3_mult}), but
modeled by (\ref{eq:data_sde4_mult}) and (\ref{eq:state_space_sde4_mult}).
Assume that (H1)--(H10), (H12) -- (H15) of MB hold (for each $i=1,\ldots,n$, whenever appropriate), with (H3) and (H7) replaced with
(H3$^\prime$) and (H7$^\prime$), respectively. Also assume (H16$^{\prime\prime}$), (H16$^{\prime\prime\prime}$) and (H17). 
Then
\begin{equation*}
\sqrt{n(b_T-a_T)}\left(\hat\theta_{n,T}-\theta_0\right)\stackrel{\mathcal L}{\longrightarrow}
N_d\left(0,{\mathcal I}^{-1}(\theta_0)\right),
\end{equation*}
as $n\rightarrow\infty$, $T\rightarrow\infty$.
\end{theorem}

\subsection{Asymptotic posterior normality in the case of multidimensional random effects}
\label{subsec:verification_posterior_normality_ss_mult}

All the conditions (1)--(7) of MB can be verified exactly as in Section \ref{subsec:verification_posterior_normality_ss}, 
only noting the appropriate multivariate extensions.
Hence, the following theorem holds:
\begin{theorem}
\label{theorem:theorem5_ss_mult}
Assume that the data was generated by the true model given by (\ref{eq:data_sde3_mult}) 
and (\ref{eq:state_space_sde3_mult}), but
modeled by (\ref{eq:data_sde4_mult}) and (\ref{eq:state_space_sde4_mult}).
Assume that (H1)--(H10), (H12)--(H15) of MB  hold (for each $i=1,\ldots,n$, whenever appropriate), 
with (H3) and (H7) replaced with
(H3$^\prime$) and (H7$^\prime$), respectively. 
Also assume (H16$^{\prime\prime}$), (H16$^{\prime\prime\prime}$) and (H17).
Then denoting $\bar\Psi_{n,T}=\bar\Sigma^{-1/2}_{n,T}\left(\theta-\hat\theta_{n,T}\right)$, for each compact subset 
$B$ of $\mathbb R^d$ and each $\epsilon>0$, the following holds:
\begin{equation*}
\lim_{n\rightarrow\infty,T\rightarrow\infty}P_{\theta_0}
\left(\sup_{\bar\Psi_{n,T}\in B}\left\vert\pi(\bar\Psi_{n,T}\vert\bar{\mathcal F}_{n,T})
-\varrho(\bar\Psi_T)\right\vert>\epsilon\right)=0.
\end{equation*}
\end{theorem}

\section{Asymptotics in the case of discrete data}
\label{sec:discrete_data}
In similar lines as \ctn{Maud12} suppose that we observe data at times $t^m_k=t_k=k\frac{b_T-a_T}{m}$; $k=0,1,2,\ldots$.
We then set
\begin{align}
v^m_{Y|X,T}&=\sum_{k=0}^{m-1}\frac{b^2_Y(Y(t_k),X(t_k),t_k)}{\sigma^2_Y(Y(t_k),X(t_k),t_k)}
\left(t_{k+1}-t_k\right)\label{eq:v_Y_discrete}\\
u^m_{Y|X,T}&=\sum_{k=0}^{m-1}\frac{b_Y(Y(t_k),X(t_k),t_k)}{\sigma^2_Y(Y(t_k),X(t_k),t_k)}
\left(Y(t_{k+1})-Y(t_k)\right)\label{eq:u_Y_discrete}\\
v^m_{X,T}&=\sum_{k=0}^{m-1}\frac{b^2_X(X(t_k),t_k)}{\sigma^2_X(X(t_k),t_k)}\left(t_{k+1}-t_k\right)
\label{eq:v_X_discrete}\\
u^m_{X,T}&=\sum_{k=0}^{m-1}\frac{b_X(X(t_k),t_k)}{\sigma^2_X(X(t_k),t_k)}
\left(X(t_{k+1})-X(t_k)\right).\label{eq:u_X_discrete}
\end{align}
For any given $T$, the actual $MLE$ or the posterior distribution can be obtained (perhaps numerically) after replacing (\ref{eq:v_Y}) -- (\ref{eq:u_X}) 
with (\ref{eq:v_Y_discrete}) -- (\ref{eq:u_X_discrete}) 
in the likelihood. 

For asymptotic inference we assume that $m=m(T)$, and that $\frac{m(T)}{b_T-a_T}\rightarrow \infty$, as $T\rightarrow\infty$. 
Then note that, since $\frac{1}{b_T-a_T}\log R_T(\theta)$ can be uniformly approximated by 
$\tilde g_T(\theta)=\tilde g_{Y,T}(\theta)+\tilde g_{X,T}(\theta)$ (as in MB) for $\theta\in\mathcal G_T\setminus I$
in the case of Bayesian consistency and for $\theta\in\Theta$ for $\Theta$ compact, for asymptotics
of $MLE$ and asymptotic posterior normality, and since $\tilde g_T(\theta)$ involve the data only through 
$(W_Y(b_T)-W_Y(a_T))/\sqrt{b_T-a_T}$, 
asymptotically the discretized version agrees with the continuous version.
This implies that, even with discretization, all our Bayesian and classical asymptotic results remain valid 
in all the $SDE$ setups considered in MB.

\normalsize
\bibliographystyle{natbib}
\bibliography{irmcmc}

\end{document}